\documentclass[10pt,psamsfonts]{amsart}

\usepackage{indentfirst}
\usepackage{graphics,amsthm}
\usepackage{color}
\usepackage{epsfig}
\usepackage{fancybox}
\usepackage{mathrsfs}
\usepackage{hyperref}
\usepackage{newlfont}
\usepackage{amsmath,amsfonts,amstext,amssymb,bezier,amsthm}
\usepackage{subfigure}
\usepackage[active]{srcltx}
\usepackage[all,cmtip]{xy}
\usepackage[english]{babel}
\usepackage[latin1]{inputenc}
\usepackage{graphicx,psfrag}
\usepackage{fancyhdr}
\usepackage{niceframe}
\usepackage{float}
\usepackage{xspace,setspace}
\usepackage{dsfont}
\usepackage{cancel}
\usepackage{ulem}
\usepackage{tikz}

\usetikzlibrary{matrix}
\usepackage{caption}

\newtheorem{obs}{Remark}[section]
\newtheorem{theorem}[obs]{Theorem}
\newtheorem{definition}[obs]{Definition}
\newtheorem{proposition}[obs]{Proposition}
\newtheorem{corollary}[obs]{Corollary}
\newtheorem{lemma}[obs]{Lemma}
\theoremstyle{definition}
\newtheorem{example}[obs]{Example}
\newtheorem{remark}[obs]{Remark}

\newcommand*\xbar[1]{%
  \hbox{%
    \vbox{%
      \hrule height 0.5pt 
      \kern0.5ex
      \hbox{%
        \kern-0.1em
        \ensuremath{#1}%
        \kern-0.1em
      }%
    }%
  }%
}

\makeindex

\begin{document}
\renewcommand{\arraystretch}{1.5}

\author[Lima]{D. V. S. Lima$^1$}

\author[Silveira]{M. R. da Silveira$^2$}

\author[Vieira]{E. R. Vieira$^3$}

\address{$^1$ CMCC, Federal University of ABC} \email{dahisy.lima@ufabc.edu.br}

\address{$^2$ CMCC, Federal University of ABC} \email{mariana.silveira@ufabc.edu.br}

\address{$^3$ DIMACS, Rutgers University} \email{ewerton.v@rutgers.edu}


\title[ Covering Action on Conley Index Theory]
{ Covering Action on  Conley Index Theory}



\maketitle

\begin{abstract}
In this paper, we  apply Conley index theory in a covering space of an invariant set $S$, possibly not  isolated, in order to describe  the dynamics in $S$.
More specifically, we consider the action of the covering translation group in order to define a topological separation of $S$
which distinguishes  the connections between the Morse sets within a Morse decomposition of $S$.
The theory developed herein generalizes the classical connection matrix theory, since one obtains enriched information on the connection maps for non isolated invariant sets, as well as, for isolated invariant sets.
 Moreover, in the case of an infinite cyclic covering induced by a circle-valued Morse function, one proves that the Novikov differential of $f$ is a particular  case of the $p$-connection matrix defined herein.

\end{abstract}

\tableofcontents

\section{Introduction}

Conley index theory is concerned about the topological structure of invariant sets  of a continuous flow on a topological space $X$, and how they are connected to each other \cite{MR511133,MR1901060,MR797044}.
The foundation of this theory, introduced in \cite{MR511133}, relies  on the fact that  there are two possibilities for the  behavior of flow lines into an isolated invariant set: a point can either be chain recurrent or it can belong to a connecting orbit from a chain recurrent piece to another one. For instance, in the case of  Morse-Novikov theory on a compact manifold,  the chain recurrent pieces are the rest points (and periodic orbits in the Novikov case) and the Morse-Novikov indices are related to the topology of the manifold. Furthermore, the Morse-Novikov inequalities impose the existence of connections between some pairs of rest points. On the other hand, in the case of Conley theory, for a flow not necessarily gradient-like, instead of connections between rest points, the global topology of the space forces connections between chain recurrent pieces (isolated invariant sets) of the flow. Such information is encoded in a matrix called connection matrix (which corresponds to the boundary operator in Morse-Novikov theory).

More specifically, given an isolated invariant set $S \subseteq X$, the approach is to consider a  decomposition $\mathcal{M}(S)$ of $S$ into a family of compact invariant sets which contains the recurrent set and such that the flow on the rest of the space is gradient-like, i.e., there is a continuous Lyapunov function which is strictly decreasing on orbits which are not chain recurrent. Such a decomposition is called a Morse decomposition of $S$ and each set of the family is known as a Morse set.  
The Conley index of each Morse set carries some topological information about the local behaviour of the flow near that set.

The connection matrix theory \cite{MR857439, MR978368, MR972705} was motivated by the desire of obtaining information on the connections between the Morse sets within a Morse decomposition. The entries of a connection matrix are homomorphisms between the homology Conley indices of the Morse sets, hence it contains information about the distribution of the Morse sets within the Morse decomposition.

In the case of a Morse-Smale flow,  connection matrices have a nice characterization. More precisely, suppose that  $\varphi$ is the negative gradient flow of a Morse function $f$ on a closed manifold $M$, satisfying the Morse-Smale transversality condition. Consider the $\prec_{f}$-Morse decomposition where each Morse set corresponds to a critical point of $f$ and $\prec_{f}$ is the flow ordering.
In this case, the connection matrix is unique and it coincides with the differential of the Morse complex as proved in \cite{MR1045282}. 

Another interesting situation is when  $\varphi$ is the negative gradient flow of a circle-valued Morse function $f$ on a closed manifold $M$ satisfying the Morse-Smale transversality condition.  One can also define a chain complex, called the Novikov complex $(\mathcal{N}_{\ast}(f),\partial^{Nov})$, as in \cite{MR2017851,pajitnov2008circle}.  However, $\partial^{Nov}$ is not a connection matrix. For instance, the  differential $\partial^{Nov}$ corresponding to the example in Figure \ref{fig:my_labelfig} is non zero. In fact, $\partial^{Nov}(h_{2}^{4},h_{1}^{2}) = \partial^{Nov}(h_{1}^{3},h_{0}^{1})  = 1 - t^2$. On the other hand, the connection matrix  is the null map.
Hence, the zero entries  of  the connection matrix  do not give information about the  connections between the corresponding  Morse sets. In this particular setting, the Novikov differential gives more information than the connection matrix.

\begin{figure}[h!t]
    \centering
    \includegraphics[scale=0.8]{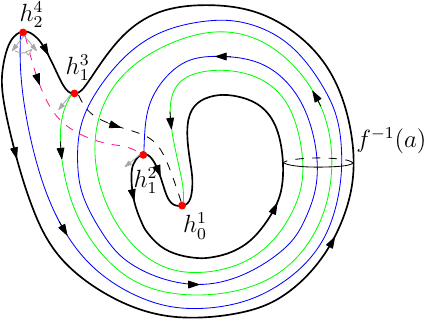}
    \caption{A flow on the torus.}
    \label{fig:my_labelfig}
\end{figure}

 One approach to enrich  the connection matrix is to consider a topological separation of the connecting sets to obtain an additive property of the connection map, as done by McCord  in \cite{MR936812}. However,  in this separation it is not possible to distinguish, for example,  the connections in Figure  \ref{fig:dif_connec}, since both have the same connection maps. Therefore,  one must  consider an algebraic structure capable of capturing more information on those connections, as the Novikov differential does.

The goal in this paper is to define a chain complex  associated to an invariant set $S$, not necessarily isolated,  whose differential  gives  enriched  information on the connections between the Morse sets of $S$.

In order to obtain information on the connecting orbits between critical points, the Novikov differential uses the Novikov ring and counts the orbits on the infinite cyclic covering on $M$. Inspired by the Novikov case, we will look for information about the connections between the Morse sets on the pullback flow defined on a regular cover $(\widetilde{M},p)$ of $M$, providing an algebraic setting that arises from the ambient space in order to distinguish those connections. More specifically,  we use the covering action to distinguish all connections up to action of the covering translation group. For instance,  the two connecting orbits in Figure \ref{fig:dif_connec} are different with respect to the covering action.

\begin{figure}[h!t]
    \centering
    \includegraphics[scale=0.9]{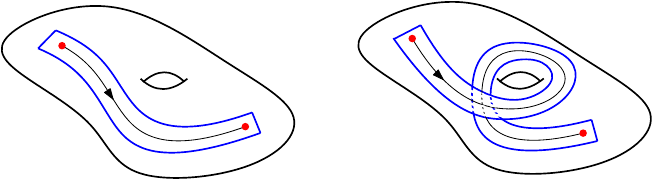}
    \caption{Flows lines on the torus.}
    \label{fig:dif_connec}
\end{figure}

We introduce a chain complex $(NC(S), N\Delta)$ associated to a pair $(S,p)$, where $S$ is an invariant set, $p$ is a regular covering map and $\mathcal{M}(S)$ is an 
  attractor-repeller decomposition of $S$. We will assume coefficients in $\mathbb{Z}((G))$,  where $G$ is the group of covering translations of $p$. The map $N\Delta$ is called a $p$-connection matrix associated to $\mathcal{M}(S)$ and it contains enriched information on the connecting orbits.  One proves the invariance of  this chain complex under equivalent covering spaces.

Whenever $S$  is an isolated invariant set and either $p$ is the trivial covering map or $G$ is projected into the trivial group,  we recover the usual setting of Conley index theory. In other words, the $p$-connection matrix introduced herein coincides with the classical connection matrix defined by Franzosa in \cite{MR978368}.   Moreover, when  $G$ is the  infinite cyclic  group  and $\varphi$ is the negative gradient flow of a circle-valued Morse function, one proves that the Novikov differential of $f$ is a particular  case of the $p$-connection matrix.

This paper is organized as follows:  Section \ref{sec:Background}  recalls relevant elements of the  connection matrix theory, as well as  some basic facts about the Novikov chain complex.
In Section \ref{sec:inv-prop}, we prove some properties of invariant sets of a pullback flow on a regular covering space.
Section \ref{sec:connectionmatrices} is at the heart of the matter, where we introduce the theory of $p$-connection matrices. In Subsection \ref{subsec:4.1}, we define $p$-attractor-repeller decompositions of invariant sets, we prove that $S$ can be decomposed into smaller invariant sets  where we can apply Conley index theory on the pullback flow.
In Subsection \ref{subsec:4.2}, we state the algebraic structure $\mathbb{Z}((G))$ that enables us  to  count the flow lines connecting the Morse sets, distinguishing  orbits  according to the deck transformation group $G$. Moreover, we present the $p$-connection matrices for invariant sets.
In Section \ref{sec:novikov}, we consider the infinite cyclic covering induced by a circle-valued Morse function, in this case $\mathbb{Z}((G))$ is the Novikov ring and the $p$-connection matrix coincides with the Novikov differential.

\section{Background}\label{sec:Background}

\subsection{Attractor-Repeller Decompositions and Connection Matrices}

Throughout this paper, let $(P,<)$ be a partial ordered set with  partial order $<$, where $P$ is a finite set of indices. An \textit{interval} in $<$ is a subset $I \subseteq P$  such that if $p,q \in I$ and $p < r < q$ then $r\in I$. The set of intervals in $<$ is denoted by $I(<)$.

An \textit{adjacent n-tuple of intervals} in $<$ is an ordered collection  $(I_1,...,I_n)$ of mutually disjoint nonempty intervals in $<$ satisfying: 
\begin{itemize}
\item $\bigcup_{i=1}^n I_i\in I(<);$
\item $\pi\in I_j, \pi'\in I_k,$ $j<k$ imply $\pi'\nless \pi.$
\end{itemize}

The collection of adjacent $n$-tuples of intervals in $<$ is denoted $I_n(<)$. An adjacent 2-tuple of intervals is also called an \textit{adjacent pair} of intervals. If $<'$ is either an extension of $<$ or a restriction of $<$ to an interval in $<$, then $I_n(<')\subseteq I_n(<)$. If $(I,J)$ is an adjacent pair (2-tuple) of intervals, then $I\cup J$ is denoted by $IJ$. If $(I_1,\ldots , I_n)\in I_n(<)$ and $\bigcup_{i=1}^n I_i=I$, then $(I_1,\ldots , I_n)$ is called a decomposition of $I$.

Let $\varphi$ be a continuous flow on a locally compact Hausdorff space $X$ and let $S\subseteq X$ be an invariant set under $\varphi$. We use the notation $x\cdot t:=\varphi(x,t)$.
For any set $Y\subseteq S$, the $\omega$-limit and $\alpha$-limit sets are given by
$\omega (Y)=\cap_{t>0}\overline{Y\cdot[t,\infty)}$ and
$\alpha (Y)=\cap_{t>0}\overline{Y\cdot(-\infty,-t]}$,  respectively. Both sets are closed, and if $S$ is compact then they will be compact.
An invariant set $A\subseteq S$ is an attractor in $S$ if
there exists a $S$-neighborhood $U$ of $A$ such that $\omega (U)=A$.  A repeller in $S$ is an invariant set $R\subseteq S$ such that
there exists a $S$-neighborhood $U$ of $R$ with $\alpha (U) = R$. Whenever $S$ is a compact set then $A$ and $R$ are also compact sets.

A ($<$-ordered) \textit{Morse decomposition} of $S$ is a collection $
\mathcal{M}(S)=\{M(\pi)\ |\ \pi\in P\}
$ of mutually disjoint compact invariant subsets of $S$, 
indexed by a finite set $P$, such that if
 $x\in S\backslash \bigcup_{\pi\in P}M(\pi)$ then there exist $\pi<\pi'$ such that $\alpha(x)\subseteq M(\pi')$ and  $\omega(x)\subseteq M(\pi).$
Each set $M(\pi)$ is called a \textit{Morse set}. A partial order on $P$ with this property induces a partial order on $\mathcal{M}(S)$  called an \textit{admissible ordering} of the Morse decomposition.

The flow defines an
admissible ordering on $\mathcal{M}(S)$, called the \textit{flow ordering} of $\mathcal{M}(S)$, denoted $<_F$, such
that $M(\pi)<_F M(\pi') $ if and only if there exists a sequence of distinct elements
of $P: \pi = \pi_0,\ldots,\pi_n=\pi'$, where the set of connecting orbits between $M(\pi_j)$ and $M(\pi_{j-1})$
$$C(M(\pi_j),M(\pi_{j-1}))=\{x \in S\backslash (M(\pi_j)\cup M(\pi_{j-1})) \mid  \alpha(x)\subseteq M(\pi_j) \text{ and } \omega(x)\subseteq M(\pi_{j-1})\}$$
 is nonempty for each $j = 1, \ldots ,n$. Note that every admissible ordering of $M$ is an extension of $<_F$.

Given a Morse decomposition $\mathcal{M}(S)$ of $S$, the existence of an admissible ordering on $\mathcal{M}(S)$ implies that any recurrent dynamics in $S$ must be contained within the Morse sets, thus the dynamics off the Morse sets must be gradient-like. For this reason, Conley index theory refers to the dynamics within a Morse set as local dynamics and off the Morse sets as global dynamics.

We briefly introduce the Conley index of an isolated invariant set and the connection matrix theory, which addresses this latter aspect.
Recall that   $S\subseteq X$ is an  \textit{isolated invariant set} if there exists a compact set $N\subseteq X$ such that $S \subseteq \mathrm{int} (N)$ and 
$$
S=Inv(N,\varphi)=\{x\in N|\ \varphi(\mathbb{R},x) \subseteq N\}.
$$
In this case $N$ is said
to be an \textit{isolating neighborhood} for $S$ in $X$. Note that isolated invariant sets are compact sets.
An \textit{index pair} for an isolated invariant set  $S$ is  a pair $(N_1, N_0)$  of compact sets $N_0 \subseteq N_1$ such that: 
(i) $S \subseteq \mathrm{int}(N_1\backslash N_0)$ and $\textrm{cl}(N_1\backslash N_0)$  is an isolating neighborhood for $S$;
(ii) $N_0$ is {\it positively invariant} in $N_1$, i.e., given $x \in N_0$ and $t>0$ such that $x \cdot[0, t] \subseteq N_1$, then $x\cdot[0,t]\subseteq N_0$;
(iii) $N_0$ is an {\it exit set} for $N_1$, i.e., given $x \in N_1$ and $t_1 > 0$ with $x\cdot t_1 \notin N_1$, there exists $t_0\in [0, t_1] $ such that $x\cdot [0, t_0]\subseteq N_1$ and $x\cdot  t_0\in N_0$. 
The theorems of existence and equivalence of index pairs guarantee that given any isolating neighborhood  $N \subseteq X$  of   $S$ and  
any neighborhood $U$  of $S$, there exists an index pair $(N_1, N_0)$ for $S$ in $X$ such
that $N_1$ and $N_0$ are positively invariant in $N$ and $cl(N_1 \setminus N_0) \subseteq U$. Moreover,
the homotopy type of the pointed space $N_1/N_0$ is independent of the choice of the
index pair and therefore it only depends on the behavior of the flow near the isolated
invariant set $S$.
For more details, see \cite{MR511133,MR797044}.

The \textit{homology Conley index} of $S$, $CH_\ast(S)$, is the homology of the pointed space $N_1/N_0$, where $(N_1,N_0)$ is an index pair for $S$.
Setting $$M(I)=\bigcup_{\pi\in I}M(\pi)\cup\bigcup_{\pi,\pi'\in I}C(M(\pi'),M(\pi)),$$ the Conley index $CH_\ast(M(I))$ of $M(I)$, in short $H_\ast(I)$, is well defined, since $M(I)$ is an isolated invariant set for all $I\in I(<)$. For more details, see \cite{MR857439}.

The simplest case of a Morse decomposition of a compact invariant set $S$ is an attractor-repeller pair $(A, R)$:  $A$ is an attractor in $S$ and $R= \{x\in S \mid \omega(x) \cap A= \emptyset\}$ is its dual repeller.  Note that, since S is compact then the dual repeller is in fact a repeller, see \cite{MR797044}. Then $S$ is decomposed into $A\cup C(R,A) \cup R$.

Given  an attractor-repeller pair $(A,R)$ of an isolated invariant set $S$, one obtains a long exact sequence, called the attractor-repeller sequence,
which relates the Conley indices of the isolated invariant sets $S, A$ and $R$, namely  $$
\cdots \longrightarrow CH_k(A) \longrightarrow CH_k(S) \longrightarrow CH_k (R) \stackrel{\partial}{\longrightarrow} CH_{k-1}(A) \longrightarrow \cdots.
$$

The map $\partial$, in the previous sequence, is called the {\it connection homomorphism} or {\it connection map}. It has the property that if $\partial\neq 0$ then there exist connecting orbits from $R$ to $A$
in $S$. In many cases, it can give more information about the set of connecting orbits. For instance, if $A$ and $R$ are hyperbolic fixed points of indices $k$ and $k-1$, respectively, satisfying the transversality condition, then the connection map is equivalent to the intersection number between the stable and unstable manifolds of $A$ and $R$, respectively.

For a Morse decomposition $\mathcal{M}(S)$ with an admissible order $(P,<)$, there is an attractor-repeller sequence for every adjacent pair of intervals in $P$.
Franzosa introduced in  \cite{MR978368}
 connection matrices as devices that allow us to encode
simultaneously  the information in all of these sequences. Roughly speaking,  connection matrices are boundary maps defined on the sum of the homology Conley indices of the Morse sets enabling each attractor-repeller sequence to be reconstructed.

 More specifically, consider an upper triangular boundary map $\Delta(P):C_\ast \Delta(P)\rightarrow C_\ast \Delta(P)$ with respect to the partial order $<$. For each interval $I\subseteq P$, set $C_\ast \Delta(I)=\bigoplus_{\pi\in I}CH_\ast (M(\pi))$ and let $\Delta(I)$ be the submatrix of $\Delta(P)$ with respect to the interval $I$. Given an adjacent pair of intervals $I, J$ in $P$, one can construct the following commutative  diagram 
$$\xymatrix{
0 \ar[r] & C_\ast \Delta(I) \ar[r]^i \ar[d]^{\Delta(I)} & C_\ast \Delta(IJ) \ar[r]^p \ar[d]^{\Delta(IJ)}& C_\ast \Delta(J) \ar[r] \ar[d]^{\Delta(J)}& 0\\
0 \ar[r] & C_\ast \Delta(I) \ar[r]^i  & C_\ast \Delta(IJ) \ar[r]^p & C_\ast \Delta(J) \ar[r] & 0\\
}$$
where $i$ and $p$ are the inclusion and projection homomorphisms, respectively. In other words,  one has a short exact sequence of chain complexes where $\Delta$ act as boundary homomorphisms.  
Since 
$(C_\ast \Delta(I),\Delta(I))$ is a chain complex, applying the homological functor $H$, the previous diagram produces a long exact sequence
$$\xymatrix{
\cdots \ar[r]& H_k\Delta(I) \ar[r]^{i_\ast} & H_k\Delta(IJ) \ar[r]^{p_\ast} & H_k\Delta(J) \ar[r]^{[\Delta(J,I)]\ \ } & H_{k-1}\Delta(I) \ar[r] & \cdots .\\
}
$$

Therefore, for every adjacent pair of intervals, the upper triangular boundary map $\Delta(P)$ generates a long exact sequence. $\Delta(P)$ is called a {\it connection matrix} if all  these sequences are canonically isomorphic to the corresponding attractor-repeller sequences. In other words, for each interval $I$, there is an isomorphism $\phi(I):H\Delta(I)\rightarrow CH(M(I))$ such that: $\phi(\{p\})=\textrm{Id}$ for every $p\in P$; and for every adjacent pair of intervals ($I, J$) the following diagram commutes 
$$\xymatrix{
 \cdots \ar[r] & H_k \Delta(I) \ar[r]^{i_\ast} \ar[d]^{\phi(I)} & H_k \Delta(IJ) \ar[r]^{p_\ast} \ar[d]^{\phi(IJ)}& H_k \Delta(J) \ar[r]^{\partial} \ar[d]^{\phi(J)}& H_{k-1}\Delta(I) \ar[d]^{\phi(I)} \ar[r]& \cdots\\
 \cdots \ar[r] & CH_k (M(I)) \ar[r]^{i_\ast}  & CH_k (M(IJ)) \ar[r]^{p_\ast} & CH_k (M(J)) \ar[r]^{\partial} & CH_{k-1}(M(I)) \ar[r]& \cdots\\
}$$

 Franzosa proved in \cite{MR978368} that, given a Morse decomposition $\mathcal{M}(S)$ of an isolated invariant set $S$, there exists a connection matrix for $\mathcal{M}(S)$. Moreover, he showed that nonzero entries in a connection matrix imply the existence of  connecting orbits, that is, if $\Delta(p,q)\neq 0$ then $p<q$, in particular, for the flow defined order $<$ there is a sequence of connecting orbits from $M(q)$ to $M(p)$.

\subsection{Dynamical Chain Complexes }\label{subsec:Nov}

In this subsection we present some  background material on dynamical chain complexes associated to Morse-Smale functions and  to circle-valued Morse functions. The main references for Morse chain complexes are \cite{banyaga2013lectures,  MR1045282, MR2243274} and for Novikov complexes are \cite{ MR2017851, pajitnov2008circle, MR1937024}.

\subsubsection{Morse chain complex}

A {\it Morse-Smale function} $(f,g)$ on a compact manifold $(M,\partial M)$ with boundary (possible empty)
 is a function $f : M \rightarrow \mathbb{R}$   together with a Riemannian metric $g$
 such that
 \begin{enumerate}
 \item  the  critical points are nondegenerate;
 \item $f$ is regular on each boundary component $N$ of  $\partial M$,  i.e.  for all $x \in N$,
$ \nabla f(x) \notin T_xN \subseteq T_xM $;
 \item for any two critical points $p, q \in M$, the stable and unstable manifolds $W^u(p)$ 
and $W^s(q)$ w.r.t the negative gradient flow $\varphi$ of $f$ intersect transversely. 
 \end{enumerate}

Let $Crit_{k}(f)$ be the set of critical points of $f$ with Morse index $k$. Given $p \in Crit_{k}(f)$ and $q\in Crit_{\ell}(f)$, define $\mathcal{M}_{pq} = W^u(p) \cap W^s(q)$, the {\it connecting manifold} of $p$ and $q$ w.r.t  $\varphi$, and 
$\mathcal{M}_p^q = W^u(p) \cap W^s(q) \cap f^{-1}(a)$, the {\it moduli space} of $p$ and $q$, i.e. the space of connecting orbits from $p$ to $q$, where $a$ is some regular value of $f$ with $f(q) < a< f(p)$.
It is well known that $\mathcal{M}_p^q$ is a $(k\!-\!\ell\!-\!1)$-dimensional manifold. Moreover, when $\ell=k-1$, $\mathcal{M}_p^q$ is a zero-dimensional compact manifold, hence it is a finite set.

Fix orientations of $T_p(W^u (p))$, for all $p\in Crit(f)$. Since $W^u(p)$ is contractible, these orientations induce
orientations on the tangent spaces to the whole  unstable manifolds. Also, since $W^s(p)$ is contractible, then the normal space  $ {\mathcal V}_p W^{s}(p)$ is orientable and the orientation of $W^u(p)$ induces an obvious orientation on $ {\mathcal V}_p W^{s}(p)$. 
 Moreover, given $p,q\in Crit(f)$, the transversality  condition implies that 
$T_{{\mathcal M}_{pq}}W^{u}(p)$ splits along ${\mathcal M}_{pq}$ as 
$  T_{{\mathcal M}_{pq}}W^{u}(p) \simeq T{\mathcal M}_{pq} \oplus {\mathcal V}_{{\mathcal M}_{pq}}W^{s}(q),$ where the last term denotes the normal bundle of $W^s(q)$ restricted to $\mathcal{M}_{pq}$. Choose an orientation on $\mathcal{M}_{pq}$  such that this  isomorphism is orientation preserving. Whenever $p\in Crit_{k}(f)$ and $q\in Crit_{k-1}(f)$, the orientation on $\mathcal{M}_{pq}$ gives an orientation on the flow line associated to each $z \in \mathcal{M}_p^q$. In this case, define
 $\epsilon(z):=+1$ if this orientation coincides with the one induced by the flow, otherwise define  $\epsilon(z):=-1$.
Finally, let
$$n(p, q;f) =  \sum_{ z\in \mathcal{M}_p^q}
\epsilon(z) .$$

Given a Morse-Smale function $(f, g) : M \rightarrow \mathbb{R}$, the $\mathbb{Z}$-coefficient {\it Morse group}
is the free $\mathbb{Z}$-module $C_{\ast}(f)=\{C_{k}(f)\}$ generated by the critical points of $f$ and graded by their Morse index, i.e, 
$C_{k}(f) = \mathbb{Z}[Crit_k(f)]$.

The $\mathbb{Z}$-coefficient {\it Morse boundary operator} $\partial$ of $f$  is defined on a generator $p$ by
\begin{eqnarray}
    \partial_k  :   C_{k}(f) & \longrightarrow &  C_{k-1}(f) \nonumber \\ 
     p  & \longmapsto &  \sum_{q \in Crit_{k-1}(f)}n(p, q;f) q . \nonumber
\end{eqnarray}

The pair $(C_{\ast}(f),\partial_{\ast})$ is called  the {\it Morse chain complex} of the Morse-Smale function $(f,g)$.

Salamon proved in \cite{MR1045282} that the Morse boundary operator is a special case of connection matrix. More specifically, considering the $<_{\varphi_f}$-ordered Morse decomposition $\mathcal{M}(M)=\{M_{\pi}\}_{\pi\in P}$ where each Morse set $M_{\pi}$ is a critical point of $f$ and $<_{\varphi_f}$ is the flow ordering, there exists a unique connection matrix for $\mathcal{M}(M)$, which coincides with the Morse boundary operator $\partial$.

\subsubsection{Novikov Chain Complex}\label{sec:Novikov}

Let $\mathbb{Z}[t,t^{-1}]$ be the Laurent polynomial ring. The {\it Novikov ring}  $\mathbb{Z}((t))$ is the  set consisting of all  Laurent series 
$$ \lambda= \sum_{i\in \mathbb{Z}}a_{i} t^{i}  $$
in one variable with coefficients $a_{i} \in \mathbb{Z}$, such that the part  of $\lambda$ with negative exponents  is finite, i.e., there is $n=n(\lambda)$ such that $a_k=0$ if $k<n(\lambda)$. In fact, $\mathbb{Z}((t))$ has a natural Euclidean ring structure such that the inclusion $\mathbb{Z}[t,t^{-1}] \subseteq \mathbb{Z}((t))$ is a homomorphism. 

Let $M$ be a {compact} connected manifold  and $f:M\rightarrow S^{1}$ be a smooth map.
Given a point $x\in M$ and a neighbourhood $V$ of $f(x)$ in $S^{1}$ diffeomorphic to an open interval of $\mathbb{R}$, the map  $f|_{f^{-1}(V)}$ is identified to a smooth map from $f^{-1}(V)$ to $\mathbb{R}$. Hence, one can  define non-degenerate critical points and Morse indices in this context as in the classical case of smooth real-valued functions. A smooth map $f:M\rightarrow S^{1}$ is called a {\it circle-valued Morse function} if its critical points are non-degenerate. Denote by $Crit(f)$ the set of critical points of $f$ and by $Crit_{k}(f)$ the set of critical points of $f$ of index $k$.

 Consider the exponential function $Exp: \mathbb{R} \rightarrow S^{1}$ given by $t\mapsto \mathrm{e}^{2\pi i t}$.    
The structure group of this covering is the subgroup $\mathbb{Z} \subseteq \mathbb{R}$ acting on $\mathbb{R}$
by translations. It is convenient to use the multiplicative notation for the
structure group and denote by t the generator corresponding to $-1$ in the
additive notation. 
 Let $(\overline{M},p_E)$ be the infinite cyclic covering  of $M$, where
  $\overline{M} = f^{\ast}(\mathbb{R}) = \{ (x,t) \in M\times \mathbb{R} \mid f(x) =[t] \in S^{1}  \} $ and
 $p_E:\overline{M}\rightarrow M$ is  induced by the map $f:M\rightarrow S^{1}$ from the universal covering $Exp$. 
 There exists a $\mathbb{Z}$-equivariant Morse-Smale function $F:\overline{M}\rightarrow \mathbb{R}$ which makes the following diagram commutative:
  \vspace{-0.3cm}
 \begin{equation}\nonumber
   \xymatrix{
  \overline{M} \ar[r]^-F \ar[d]_{p_E} & \mathbb{R} \ar[d]^{Exp} \\
  M \ar[r]^-f  & \mathbb{S}^1
  }
 \end{equation}
Note that if $Crit(F)$ is nonempty then it  has  infinite cardinality. Since $\overline{M}$ is noncompact, one can not apply the classical Morse theory to study $F$. To overcome this, one can restrict $F$  to a fundamental cobordism $W$ of $\overline{M}$ with respect to the action of $\mathbb{Z}$.
 The {\it fundamental cobordism} $W$ is defined as 
$ W = F^{-1}([a-1,a]),  $
 where $a$ is a regular value of $F$. It can be viewed  as the compact manifold obtained by cutting $M$ along the submanifold $V=f^{-1}(\alpha)$, where $\alpha = Exp(a)$. Hence, $(W, V,t^{-1}V)$ is a cobordism with both boundary components diffeomorphic to $V$.

 From now on, we consider circle-valued Morse functions $f$ such that  the vector field $-\nabla f$  satisfies the transversality condition, i.e., the lift $-\nabla F$  of $-\nabla f$ to $\overline{M}$ satisfies the classical transversality condition on the unstable and stable manifolds. Denote by $\overline{\varphi}$ the pullback of $\varphi$, where $\varphi$ is the flow associated to  $-\nabla f$.

Fix $\overline{p},\,\overline{q}\in Crit(F)$ lifts of $p, q\in Crit(f)$, respectively.
Choosing arbitrary orientations for all unstable manifolds $W^{u}(p)$ of critical points of $f$, one considers the  induced orientations on the unstable manifolds $W^{u}(t^{\ell}\overline{p})$ and $W^{u}(t^{\ell}\overline{q})$, for $\ell \in \mathbb{Z}$. 
As each path in $M$ that originates
at $p$ lifts to a unique path in $\overline{M}$ with origin $\overline{p}$, the space $\bigcup_{\ell\in\mathbb{Z}}\mathcal{M}({\overline{p},t^{\ell}\overline{q}})$ of flow lines of $\overline{\varphi}$ that join $\overline{p}$ to one of the points $t^{\ell}\overline{q}$, $\ell\in\mathbb{Z}$, is homeomorphic to $\mathcal{M}(p,q)$. In
particular, for $p$ and $q$ consecutive critical points, by the equivariance of $F$, $$n(t^{\ell}\overline{p},t^{\ell}\overline{q};F)=n(\overline{p},\overline{q};F)$$ for all $\ell\in\mathbb{Z}$, 
 where $n(\overline{p},t^{\ell}\overline{q};F)$ is the intersection number between the critical points $\overline{p}$ and $t^{\ell}\overline{q}$ of $F$.

Given $p \in Crit_{k}(f)$ and $q \in Crit_{k-1}(f)$, {\it the Novikov incidence coefficient} between $p$ and $q$ is defined as 
$$ N(p,q;f) = \sum_{\ell \in \mathbb{Z}}n(\overline{p},t^{\ell}\overline{q};F) t^{\ell}. $$
For more details, see \cite{MR2017851} and \cite{pajitnov2008circle}. 

Let $\mathcal{N}_{k}$ be  the $\mathbb{Z}((t))$-module freely generated by the critical points of $f$ of index $k$. Consider the $k$-th boundary operator $\partial^{Nov}_{k}: \mathcal{N}_{k} \rightarrow \mathcal{N}_{k-1}$ which is  defined on a generator $p \in Crit_{k}(f)$ by 
$$ \partial^{Nov}_{k}(p) = \sum_{q\in Crit_{k-1}(f)} N(p,q;f)q $$
and extended to all chains by linearity. In \cite{pajitnov2008circle} it is proved that $\partial^{Nov}_{k}\circ\partial^{Nov}_{k+1}=0$, hence $(\mathcal{N}_{\ast}(f),\partial^{Nov})$ is a chain complex which is called the {\it Novikov chain complex} associated to $f$.

\section{Invariance Properties of Pullback Flows on Covering Spaces}\label{sec:inv-prop}

\vspace{0.4cm}

Consider a metric space $X$ which admits a regular covering space $\widetilde{X}$ with covering map $p:\widetilde{X}\to X$  and let $G$ be the group of the covering translations (deck transformation group).  Thus, the action of $G$ on each fiber is free and transitive and the quotient $\widetilde{X} / G$ can be identified with $X$.
Given a subset $B$ of $X$ and $g\in G$, we will denote by $gB$ the set $\{gb \mid  b\in B\}$ and  if $e\in G$ is the trivial element, $B=eB$.

Let $\varphi: \mathbb{R} \times X \rightarrow X$ be a continuous flow on $X$.
If $\widetilde{X}$ is connected, locally path connected and $\varphi_{\#}\circ (Id\times p)_{\#}(\pi_1(\mathbb{R}\times \widetilde{X}))\subseteq p_{\#}(\pi_1(\widetilde{X}))$, then one can define the pullback flow of $\varphi$ by $p$, denoted by $\widetilde{\varphi}$, as the lifting of the map $\varphi\circ (Id\times p)$. Hence, one has the following  commutative diagram: 
\xymatrix{}
$$
\xymatrixcolsep{3pc}\xymatrix{
(\mathbb{R}\times \widetilde{X},(0,\tilde{x}))\ar[r]^{\ \ \ \ \widetilde{\varphi}} \ar[d]^{Id\times p} & (\widetilde{X},\tilde{x}) \ar[d]^{p}\\
(\mathbb{R}\times {X},(0,x)) \ar[r]^{ \ \ \ \ \varphi} & (X,x)
}
$$

Note that, as a consequence of the unique path lifting property of coverings,
when $p$ restricts to a homeomorphism from some subset $\widetilde{Y}$ of $\widetilde{X}$ onto an invariant subset of $X$, then $\widetilde{Y}$ is also an  invariant set.

If $x\in X$  and $\varphi(\mathbb{R},x)$  is an aperiodic orbit, then the trajectories of the points
of $p^{-1}(x)$ under the flow $\tilde{\varphi}$ are pairwise disjoint and aperiodic, and $p$ restricted
to any such trajectory is one-to-one. See \cite{MR1823962}.

Throughout this paper, let $X$ be  a locally compact metric space and $(\widetilde{X},p)$ be a  regular cover of $X$, where $\widetilde{X}$ is a connected, locally path connected metric space. 
Also, we use the following definition:  a set $U \subseteq X$ is {\it evenly covered} by $p$ if $p^{-1}(U)$ is a disjoint union of sets $\widetilde{U}_\lambda \subset \widetilde{X}$ such that $p_{|\widetilde{U}_\lambda}:\widetilde{U}_\lambda \rightarrow U$ is a  homeomorphism for every $\lambda$.
The homeomorphic copies in $\widetilde{X}$ of an evenly covered set $U$ are called  {\it sheets} over $U$.

The next result is  a characterization of evenly covered sets. 

\begin{proposition}\label{prop:evenly}
Let $S\subseteq X$ be  an evenly covered set. Then there exists $ \widetilde{S} \subset \widetilde{X}$ such that
 $p^{-1}(S)=\displaystyle\bigsqcup_{g\in G} g\widetilde{S}$  and  $p|_{g\widetilde{S}}$ is a homeomorphism, where $G$ is the deck transformation group.
\end{proposition}
\begin{proof} 
Since $S$ is an evenly covered set, there exists a sheet $\widetilde{S} \subset \widetilde{X}$ over $S$ such that  
$p|_{\widetilde{S}}:\widetilde{S}\to S$  is a homeomorphism.
It follows from the  freeness and the transitivity of the action of $G$ in $\widetilde{X}$ that $p^{-1}(S)=\displaystyle\bigsqcup_{g\in G} g\widetilde{S}$.
Moreover,  given a deck transformation  $g\in G$, one  has that 
$p|_{g\widetilde{S}}
= p|_{\widetilde{S}} \circ g^{-1}|_{g\widetilde{S}}$ is a homeomorphism.
\end{proof}

The next  result gives an important property of evenly covered sets on a regular covering space which is essential along this work. 

\begin{theorem}\label{prop_W}
Let $S\subseteq X$ be  an evenly covered compact set. Given a sheet $\widetilde{S}$ over $S$,
there exists a  neighborhood $W$ of $\widetilde{S}$ such that $p_{|_W}$ is a homeomorphism onto its image. 
\end{theorem}
	
	\begin{proof}

It is sufficient  to prove that there exists a compact neighborhood $W$ of $\widetilde{S}$ such that  $p_{|_W}$ is injective.
	
By Proposition \ref{prop:evenly}, $p^{-1}(S) = \bigsqcup_{g\in G} g\widetilde{S}$. Let $\mathcal{B}=\Big( \bigsqcup_{g\in G} g\widetilde{S} \Big) \setminus \widetilde{S}$. First,  we prove that there exist closed  disjoint neighborhoods of $ \mathcal{B}$ and $\widetilde{S}$.
	
	\vspace{0.2cm}
	
\noindent	\underline{Claim 1:} $\mathcal{B}$ is a closed set.

If $\mathcal{B} =\emptyset$, the claim holds. Suppose that $\mathcal{B} \neq\emptyset$. Let $x\in \overline{\mathcal{B}} $ and let  $(x_n)_{n\in \mathbb{N}} $ be a sequence in $\mathcal{B}$ such that $x_n \rightarrow x$. Then  $(p(x_n))_{n\in \mathbb{N}} $ is a sequence in $S$ and $p(x_n)\to p(x)$. Since $p^{-1}(S)$ is a closed set, then $x\in p^{-1}(S)$. Suppose $x\in \widetilde{S}$. Let $V_x$ and $U_{p(x)}$ be  neighborhoods of $x$ and $p(x)$, respectively, such that $p:V_x\to U_{p(x)}$ is a homeomorphism.
Then there is $n_0 \in \mathbb{N}$ such that $x_n\in V_x$ for all $n>n_0$. Moreover, since $p_{|_{\widetilde{S}}}:\widetilde{S}\to S$ is a homeomorphism, there exists a sequence $(x_n^{\prime})_{n\in \mathbb{N}}$ in $\widetilde{S}$ such that $p(x_n)=p(x_n^{\prime})$ and  $x_n^{\prime} \rightarrow x$. Then there is $n_1 \in \mathbb{N}$ such that $x_n^{\prime}\in V_x$ for all $n>n_1$.
Thus for all $n>\mbox{max}\{n_0,n_1\}$, $x_n \neq x_n^{\prime}$ and $p(x_n)= p(x_n^{\prime})$.
This contradicts the fact that $p_{{|}_{V_x}}$ is a homeomorphism, hence $x\notin \widetilde{S}$. Consequently, $x\in \mathcal{B}$ and $\mathcal{B}$ is a closed set. $\square$

As a consequence of Claim 1 and the normality of $\widetilde{X}$, there exist closed neighborhoods  $F_{\widetilde{S}}$ and $F_{\mathcal{B}}$ of $\widetilde{S}$ and $ \mathcal{B}$, respectively, such that $F_{\widetilde{S}} \cap F_{\mathcal{B}} =\emptyset$.

Since $\widetilde{X}$ is locally  compact, for each $\tilde{s}\in \widetilde{S}$, there is  a  compact neighborhood   $V_{\tilde{s}}$  of $\tilde{s}$
such that $V_{\tilde{s}}\subset F_{\widetilde{S}}$ and $U_s:=p(V_{\tilde{s}})$ is an evenly covered neighborhood of $s:=p(\tilde{s})$. By  the compactness  of $\widetilde{S}$, there are $\tilde{s}_1,\ldots,\tilde{s}_{\ell} \in \widetilde{S}$ such that $ \{\text{int}(V_{\tilde{s}_i})\}_{i=1}^{\ell}$ is a finite open cover  of $\widetilde{S}$ and hence $ \{V_{\tilde{s}_i}\}_{i=1}^{\ell}$ is a finite compact cover  of $\widetilde{S}$, which will be denoted by $\{V_i\}_{i=1}^{\ell}$.  Note that, the correspondence between the collections $\{U_i\}_{i=1}^{\ell}$ and $\{V_i\}_{i=1}^{\ell}$ is bijective given that $\widetilde{S}$ is homeomorphic to $S$ via $p$.

Consider the sets  $A_{ij}=\{y\in V_i\backslash V_j \ | \ \exists x\in V_j \text{ such that } p(x)=p(y)\}$, for $i,j \in \{1,\ldots,\ell\}$. Since $A_{ij}\subset F_{\widetilde{S}}$ then $A_{ij}\cap \widetilde{S}=\emptyset$.

\noindent \underline{Claim 2:}  $\widetilde{S} \cap \overline{A}_{ij}=\emptyset$, for all $i,j$.
	
 In fact,  suppose  $\overline{A}_{ij}\cap \widetilde{S} \neq \emptyset$, thus there is a sequence $(y_n)_{n\in \mathbb{N}}$ in $ A_{ij}$ such that $y_n \rightarrow \tilde{s}$, for some $\tilde{s}\in \widetilde{S}$. 
 By the  definition of $A_{ij}$, there exists a sequence $(x_n)_{n\in \mathbb{N}}$ in $V_j$ such that $x_n\neq y_n$ and $p(x_n)=p(y_n)$, hence
$p(y_n) = p(x_n) \rightarrow p(\tilde{s})$. Since $(x_n)_{n\in \mathbb{N}}$ is a sequence in the compact set $V_j$, taking a subsequence if  necessary, one can assume that $(x_n)_{n\in \mathbb{N}}$ convergences to a certain $b\in p^{-1}(s)$, where $s:=p(\tilde{s})$. Since $b\in V_j\subset F_{\widetilde{S}}$ then $b=\tilde{s}$. 
Since $p$ is locally injective and both sequences $(y_n)_{n\in \mathbb{N}}$ and $  (x_n)_{n\in \mathbb{N}}$ converge to $\tilde{s}$, there exists $n_0\in\mathbb{N}$ such that $x_n=y_n$ for all $n>n_0$. This contradicts the fact that $x_n\neq y_n$, for all $n$. Therefore  $\widetilde{S} \cap \overline{A}_{ij}=\emptyset$, for all $i,j$.
$\square$
	
As a consequence of Claim 2 and the normality of $\widetilde{X}$, there exist closed neighborhoods  $F_{ij}^{\widetilde{S}}$ and $F_{ij}^{A}$ of  $\widetilde{S}$ and $ \overline{A}_{ij}$, respectively, such that $F_{ij}^{\widetilde{S}} \cap F_{ij}^{A} =\emptyset$.	
	
Finally, consider $W=(\cup_{i=1}^{\ell} V_i)\cap(\cup_{i,j=1}^{\ell} F_{ij}^{\widetilde{S}})$\ the compact neighborhood of $\widetilde{S}$.	
Since $W$ is compact and $p_{|_{W}}$ is injective, then $p_{|_{W}}$ is a homeomorphism.
\end{proof}

 Although in the proof of Theorem \ref{prop_W} one extends the homeomorphism $p:\widetilde{S} \rightarrow S$  to a  compact neighborhood  $W$ of $\widetilde{S}$, one could also extend it to an open neighborhood of $\widetilde{S}$.

The following result is a direct consequence of the previous theorem.

\begin{proposition}\label{prop:Dahisymon}

Let $S\subseteq X$ be  an evenly covered compact set. Given a sheet $\widetilde{S}$ over $S$, $S$ is an isolated invariant set iff $\widetilde{S}$ is an isolated invariant set. Moreover,  in this case the homology Conley indices of $S$ and $\widetilde{S}$ coincide, i.e. $CH_{\ast}(S) = CH_{\ast}(\widetilde{S})$.
\end{proposition}
\begin{proof}
By Theorem \ref{prop_W}, there exists a neighborhood $W$ of $\widetilde{S}$ such that $p_{|_W}$ is a homeomorphism. Since $\varphi\circ (Id\times p)= p\circ \widetilde{\varphi}$, one has that $p\circ\widetilde{\varphi}(t,\tilde{x})=\varphi(t,p(\tilde{x}))$, which implies that the flows ${\varphi}$ and $\widetilde{\varphi}$ restricted to $W$ are topologically equivalent by $p_{|_W}$.
Hence, $\widetilde{S}$ is an isolated invariant set iff $S$ is an isolated invariant set.  
The isomorphism between the homology Conley indices of $S$ and $\widetilde{S}$ follows from the existence of a neighbourhood
basis of index pairs for an isolated invariant set.
\end{proof}

 The invariant sets considered in the classical Conley theory (e.g. \cite{MR511133, MR857439, MR978368, MR797044})
are isolated invariant sets and hence compact. As a consequence, the attractors and repellers in theses sets are always compact. Since the goal in  this work is to study invariant sets which are not necessarily compact, one considers the following definitions of attractors and repellers.
Given $S$ an invariant set not necessarily compact,
a compact invariant set $A \subseteq S$ is an attractor in $S$ if  there exists an open neighborhood $U$ of $A$ in $S$ such that $A = \omega(U)$. A compact invariant set $R \subseteq S$ is a repeller in $S$ if  there exists an open neighborhood $U$ of $R$ in $S$ such that $R = \alpha(U)$.

\begin{proposition}\label{prop:Ewertonmon}

Let $S$ be an invariant set. 
Given  $A\subseteq X$   an evenly covered attractor in $S$, 
if  $\widetilde{A}$ is a sheet over $A$, then $\widetilde{A}$ is an attractor in $p^{-1}(S)$. Moreover $CH_{\ast}(A) = CH_{\ast}(\widetilde{A})$.
\end{proposition}

\begin{proof}
Since $A$ is an evenly  covered compact set, by Theorem \ref{prop_W}, there exists a neighborhood $W$ of $\widetilde{A}$ such that $p_{|_W}$ is a homeomorphism. The proof of Proposition \ref{prop:Dahisymon} shows that  the flows ${\varphi}$ and $\widetilde{\varphi}$ are topologically equivalent in $W$ by $p_{|_W}$. Therefore, $\widetilde{A}$ is an attractor in $p^{-1}(S)$.
\end{proof}

Analogously, if  $R$ is an 
evenly covered repeller in $S$, 
and  $\widetilde{R}$ is a sheet over $R$, then $\widetilde{R}$ is a repeller in $p^{-1}(S)$ and $CH_{\ast}(R) = CH_{\ast}(\widetilde{R})$.

\begin{proposition} 
Let $S\subseteq X$ be  an evenly covered compact set and  $\widetilde{S}$ a sheet over $S$.   

\begin{enumerate}
    \item  If $(A,R)$ is 
 an attractor-repeller pair of ${S}$, then $(p^{-1}(A)\cap \widetilde{S},p^{-1}(R)\cap \widetilde{S})$ is 
 an attractor-repeller pair of $\widetilde{S}$;
    
    \item If $(\widetilde{A},\widetilde{R})$ is 
 an attractor-repeller pair of $\widetilde{S}$, then $(p(\widetilde{A})\cap {S},p(\widetilde{R})\cap {S})$ is 
 an attractor-repeller pair of ${S}$.
\end{enumerate}
\end{proposition}

\begin{remark} 

Assuming that $S \subseteq X$ is an isolated invariant set does not necessarily  imply that $p^{-1}(S)$ is also an isolated invariant set.  For instance, consider the flow
on the torus $\mathbb{T}^2$ and the corresponding pullback flow on the infinite cyclic covering $\widetilde{\mathbb{T}}^2$, as in Figure  \ref{fig:exe1}. The lift of the periodic orbit  $\mathcal{O}$ is not an isolated invariant set.

\begin{figure}[h!t]
    \centering

    \includegraphics[scale=0.75]{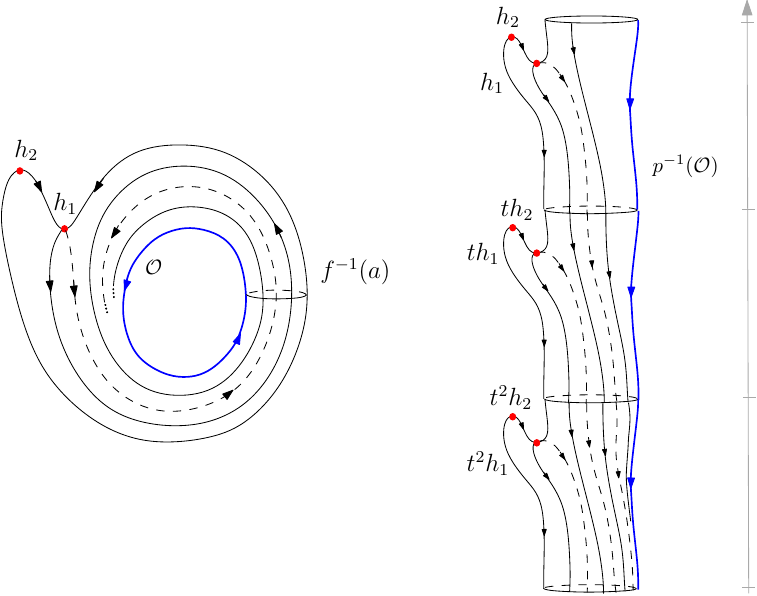}
    \caption{Infinite cyclic covering of the torus $\mathbb{T}^2$: $p^{-1}(\mathcal{O})$ is not isolated. }
    \label{fig:exe1}
\end{figure}

\end{remark}

The next proposition guarantees that all the properties of a given set $\widetilde{S} \subseteq \widetilde{X}$ which we are interested in, such as invariance and isolation,  are preserved by equivalent coverings of $X$.

\begin{proposition}\label{prop:equivariant}
Let 
$(\widetilde{X}_{i},p_i)$ be equivalent regular coverings of $X$, where $\widetilde{X}_i$ is a connected, locally path connected metric space, for $i=1,2$. Let $S\subseteq X$ be  an evenly covered set and  $\widetilde{S}$ a sheet over $S$ with respect to $p_1$.   Let $h:\widetilde{X}_2\rightarrow \widetilde{X}_1$ be a  homeomorphism which provides an equivalence of the covering spaces. 
\begin{enumerate}
\item The set $S$ is evenly covered with respect to $p_2$.
  \item If $\widetilde{S}$ is an  invariant set in $\widetilde{X}_1$, then $h^{-1}(\widetilde{S})$ is an  invariant set in $\widetilde{X}_2$.
      \item If $\widetilde{S}$ is compact and $(\widetilde{A},\widetilde{R})$ is 
 an attractor-repeller pair of $\widetilde{S}$, then $ (h^{-1}(\widetilde{A}),h^{-1}(\widetilde{R}) )$
is 
 an attractor-repeller pair of $h^{-1}(\widetilde{S})$. Moreover,  $CH_{\ast}(\widetilde{A})= CH_{\ast}(h^{-1}(\widetilde{A}))$ and $CH_{\ast}(\widetilde{R})= CH_{\ast}(h^{-1}(\widetilde{R}))$.
    \item If $\widetilde{S}$ is an isolated invariant set in $\widetilde{X}_1$, then $h^{-1}(\widetilde{S})$ is an isolated invariant set in $\widetilde{X}_2$ and $CH_{\ast}(\widetilde{S})= CH_{\ast}(h^{-1}(\widetilde{S}))$.

\end{enumerate}
\end{proposition}

\begin{proof} The proof of (1) is straightforward. Now, consider the  following diagram  (where we omit the basepoints)
\xymatrix{}
\begin{equation}
\xymatrixcolsep{3pc}\xymatrix{
\mathbb{R}\times \widetilde{X}_{2}\ar[r]^{\quad \tilde{\varphi}_{2}} \ar[d]^{Id\times h} & \widetilde{X}_{2} \ar[d]^{h}\\
\mathbb{R}\times \widetilde{X}_{1} \ar[r]^{\quad \tilde{\varphi}_{1}} & \widetilde{X}_{1}
}\nonumber
\end{equation}
This diagram is commutative. In fact, 
\begin{eqnarray}
 p_{1} \circ h\circ \tilde{\varphi}_2 \circ (Id\times h^{-1}) & =&   p_{2}\circ \tilde{\varphi}_2 \circ (Id\times h^{-1}) \nonumber\\
 & = &  \varphi \circ (Id\times p_2) \circ  (Id\times h^{-1}) \ = \ \varphi \circ (Id\times p_1), \nonumber
\end{eqnarray}
 which implies, by the uniqueness of the lifting  $\tilde{\varphi}_{1}$, that $\tilde{\varphi}_{1} = h \circ \tilde{\varphi}_{2}\circ (Id \times h^{-1})$.

If $\widetilde{S}\subseteq \widetilde{X}_1$ is an invariant set (resp., isolated invariant set) w.r.t. $\tilde{\varphi}_1$ then  $h^{-1}(\widetilde{S})$ is  also an invariant set (resp., isolated invariant set)  under $\tilde{\varphi}_2$, by the commutativity of the diagram above. This proves itens (2) and (4).

In order to prove item (3), it is sufficient to show that $\omega(h(\widetilde{U}))=h(\omega(\widetilde{U}))$ and $\alpha(h(\widetilde{U}))=h(\alpha(\widetilde{U}))$, for all $\widetilde{U}\subseteq \widetilde{X}_2$. One has that

\begin{eqnarray}
\omega(h(\widetilde{U})) & = &  \bigcap_{t>0} \overline{\tilde{\varphi}_1([t,\infty),h(\widetilde{U}))} \nonumber\\
 & = &  \bigcap_{t>0} \overline{ h \circ \tilde{\varphi}_{2}\circ (Id \times h^{-1})([t,\infty),h(\widetilde{U}))}\nonumber\\
 & = & h\left( \bigcap_{t>0} \overline{ \tilde{\varphi}_{2}([t,\infty),\widetilde{U})}\right)=h(\omega(\widetilde{U})).\nonumber
\end{eqnarray}

Analogously, one proves that $\alpha(h(\widetilde{U}))=h(\alpha(\widetilde{U}))$.
\end{proof}

\section{$p-$Connection Matrices}\label{sec:connectionmatrices}

In this section we will define a $p-$connection matrix for a $p-$Morse decomposition of an invariant set $S$. Its entries are homomorphisms which give dynamical information on the connecting orbits between  $p-$Morse sets.
In this setting, we only assume that $S$ is an invariant set (possibly noncompact), dropping the assumption that $S$ is isolated, even though the $p-$Morse sets are considered to be isolated invariant sets (hence, compact sets).

In Subsection \ref{subsec:4.1}, one defines a $p$-attractor-repeller pair $(A,R)$ for $S$ as an attractor-repeller decomposition of $S$
such that $A$ and $R$ are disjoint $p$-evenly covered isolated invariant sets, see Definition \ref{def:evenlycovered}.
Despite the fact that $S$ may not be compact or $p$-evenly covered, one proves that, under some additional hypothesis, $S$ can be decomposed  into smaller invariant sets ${S}_{{R},g{A}}$ which are compact evenly covered sets, see Theorem \ref{teo:Marimon}. 
In Subsection \ref{subsec:4.2}, one defines a $p$-connection matrix for a $p$-attractor-repeller decomposition of an invariant set and one proves its  invariance under equivalent regular covering spaces. {Moreover, for the case of isolated invariant sets, one establishes in Theorem \ref{thm_proj} the relation between \textit{p-connection matrices} and the classical connection matrices presented in \cite{MR978368}, showing that the \textit{p-connection matrix} generalizes the classical one.}
In Subsection \ref{sec:p-morsedec}, 
a manner to use this theory to obtain information on the connections between Morse sets in a more general $p$-Morse decomposition is presented.  More specifically, given  a $p$-Morse decomposition of $S$, one looks at the maps between the p-Morse sets which are adjacent.
In Subsection \ref{subsec:examples}, one presents some examples to illustrate the results obtained in the previous subsections.

\subsection{$p$-Attractor-Repeller Decomposition for Invariant Sets}\label{subsec:4.1}

It is well known that, when $S$ is compact, each orbit has nonempty $\alpha$ and $\omega$-limit sets. However, this is not always the case when $S$ is noncompact. For instance,  the flow on $\mathbb{R}^2$ as in Figure \ref{fig:5} has a flow line $\gamma$ whose $\alpha$ and $\omega$-limits are empty.  In this case, if we consider the usual definition of connection between two invariant sets, then in Figure \ref{fig:5} the orbit $\gamma$ would be a connection between $R$ and $A$. In order to discard connections of these types, we restrict our analysis to the connecting orbits that have nonempty $\alpha$ and $\omega$-limit sets.

\begin{figure}[h!t]
    \centering
    \includegraphics[scale=0.7]{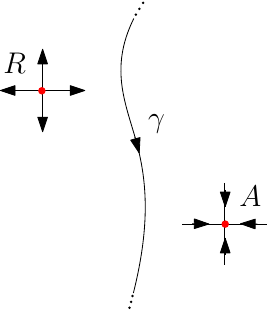}
    \caption{An orbit $\gamma$ with $\alpha(\gamma)=\emptyset$ and $\omega(\gamma)=\emptyset$ that is not a connection between $R$ and $A$.}
    \label{fig:5}
\end{figure}

Recall that  $X$ is a locally compact metric space, $\varphi: \mathbb{R} \times X \rightarrow X$ is a continuous flow on $X$ and $(\widetilde{X},p)$ is a  regular cover of $X$, where $\widetilde{X}$ is a connected, locally path connected metric space.  Let $S$ be an invariant set in $X$. Given  $A$ and $R$  invariant sets of $S$, the set of connections between $A$ and $R$ is defined by 
    $$C^{\ast}(R,A)=\{x \in S\backslash (A\cup R) \mid \alpha(x)\neq \emptyset, \omega(x)\neq \emptyset , \alpha(x)\subseteq R \text{ and } \omega(x)\subseteq A\}.$$

    \begin{definition}\label{def:evenlycovered}
Given an invariant set $S$ in $X$, a  pair of disjoint compact invariant sets $(A,R)$ is a $p$-attractor-repeller pair for $S$ if 
\begin{enumerate}
    \item  $A$ and $R$ are evenly covered; 
    \item $A$ is an attractor  in $S$ and  $R$ is a  repeller in $S$;
    \item given $x\in S$, then either $x\in A$ or $x\in R$ or $x\in C^{\ast}(R,A).$
\end{enumerate}
 The decomposition $S={R} \cup C^{\ast}({R},{A}) \cup {A}$ is called a $p$-attractor-repeller decomposition of $S$. 
\end{definition}

Given a $p$-attractor-repeller pair for $S$ and  $A^{\ast} := \{ x\in S \mid \omega(x) \cap A =\emptyset\}$, note that  $R = A^{\ast}$.

\vspace{0.2cm}

It is clear that there exist invariant sets which do not admit  $p$-attractor-repeller decompositions for any covering $p$.

\vspace{0.2cm}

 \begin{remark}\textcolor[rgb]{1,1,1}{.}
 If $S$ is  an isolated invariant set, one has that $C^{\ast}(R,A) = C(R,A)$. Hence,  a p-attractor-repeller pair of $S$ is an attractor-repeller pair of $S$ in the usual sense as defined in \cite{MR511133, MR978368, MR797044}.
    \end{remark}
 
Note that, in the previous definition, it is not required that $S$ is compact. In this paper, we are interested in an  attractor-repeller decomposition such that  the deck transformation group ``acts freely and transitively" on the attractors and repellers.
This property, which is a consequence of Proposition \ref{prop:evenly}, is stated in the next result.

\begin{proposition}
Given a $p$-attractor-repeller pair $(A,R)$ for an invariant set $S$, there exist compact invariant sets $\widetilde{A}, \widetilde{R} \subseteq \widetilde{X}$ such that
\begin{enumerate}
\item $p^{-1}(A)=\displaystyle\bigsqcup_{g\in G} g\widetilde{A}$  and  $p|_{g\widetilde{A}}$ is a homeomorphism;
\item $p^{-1}(R)=\displaystyle\bigsqcup_{g\in G} g\widetilde{R}$  and $p|_{g\widetilde{R}}$ is a homeomorphism,
\end{enumerate}
where $G$ is the deck transformation group.
\end{proposition}

It is well known that, if $S$ is a compact set and $Y\subseteq S$, then $\omega(Y)$ and $\alpha(Y)$
are compact invariant subsets of $S$. Furthermore, if $U$
is a neighborhood of $\omega(Y)$, then there exists  $t>0$ such that $Y.[t ,\infty)\subseteq U$. A
similar statement holds for $\alpha(Y)$. Whenever $S$ is not compact, these properties do not necessarily hold. See example in Figure \ref{fig:Contra_ex_4}. 
However, one can still retrieve some nice properties for subsets of $S$ which admit compact neighborhoods. The next proposition  states a property of $\alpha(Y)$ and $\omega(Y)$  when $S$ is not necessarily compact.

\begin{lemma}\label{lem:claim1}
Let $S$ and $Y$ be invariant sets in $X$ such that $Y\subseteq S$. If there exists a compact set $K$ in $X$ such that $Y\subseteq K\subseteq S$, then for every open neighborhood $U$ of $\omega(Y)$ there exists $t$ such that $Y.[t,\infty)\subseteq U$.  \end{lemma}
\begin{proof}
 Assume that the claim is false. Then there is an open neighborhood $U$ of $\omega(Y)$, a sequence
of points $y_n \in Y$  and a sequence $t_n \in \mathbb{R}$ with $t_n \rightarrow \infty$ such that $y_n . t_n \notin U$ for all $n$. Since $K$
is sequentially compact and $y_n. t_n \notin U$ for all $n$, there is a subsequence of $y_n . t_n$ that converges to some point $z \in K$ with
$z \notin \omega(Y)$. However, $z \in \overline{Y.[t,\infty)}$ for all $t > 0$ and hence
$z \in \omega(Y) =\displaystyle \bigcap_{t>0}  \overline{Y.[t,\infty)} 
$.
This contradiction establishes the result.
\end{proof}

The following theorem is a generalization of the Path Lifting Theorem for orbits which are contained in a compact set.

\begin{theorem}[Lifting of orbits]\label{teo:SVL} 
Let  $(A,R)$ be a
$p$-attractor-repeller pair for an invariant set $S$. Let $\gamma$ be an orbit of $\varphi$ such that $\alpha(\gamma)\subseteq R $,  $\omega(\gamma)\subseteq A$ and there is a compact set $U\subseteq X$ containing $\gamma$. Fixing sheets  $\widetilde{R}$ over $R$  and $\widetilde{A}$ over $A$, there exist a unique $g\in G$ and a unique orbit $\widetilde{\gamma}$ of $\widetilde{\varphi}$ such that $\alpha(\widetilde{\gamma})\subseteq \widetilde{R}$,   $\omega(\widetilde{\gamma})\subseteq g\widetilde{A}$ and $p(\widetilde{\gamma})=\gamma$.
 \end{theorem}

\begin{proof}

By Theorem \ref{prop_W}, there are open neighborhoods $\widetilde{V}_{R}$ and $\widetilde{V}_{A}$ of $\widetilde{R} $ and $\widetilde{A} $ such that $p|_{\widetilde{V}_{R}}$ and $p|_{\widetilde{V}_{A}}$ are homeomorphisms. Let $V_{R}:= p(\widetilde{V}_{R})$ and $V_{A}:=p(\widetilde{V}_{A})$. Since $\gamma:(-\infty,\infty)\to M$ is contained in a compact set, it follows from Lemma \ref{lem:claim1} that there exists $t^{\ast}$ such that $\gamma((-\infty,-t^{\ast}])\subset {V}_{R} $ and $\gamma([t^{\ast},\infty))\subset {V}_{A}$, see Fi gure \ref{fig:my_label}.

Denote by $a$ and $b$  the lifts of $\gamma(-t^{\ast})$ and  $\gamma(t^{\ast})$ which belong to  $\widetilde{V}_{R}$ and $\widetilde{V}_{A}$, respectively. Considering the  path $\gamma{|_{{[-t^{\ast},t^{\ast}]}}}$, by the unique path lifting property, there is a unique path $\widetilde{\beta}:[-t^{\ast},t^{\ast}]\rightarrow \widetilde{X}$ such that  $\widetilde\beta(-t^{\ast})=a$ and  $p\circ\widetilde\beta=\gamma$. Since $\gamma(t^{\ast}) \in V_{A}$, then $\widetilde{\beta}(t^{\ast}) \in p^{-1}(V_{A})$. Therefore, by the transitivity and freeness  of the action, there exists a unique $g \in G$ such that $\widetilde{\beta}(t^{\ast}) = gb  \in g\widetilde{V}_{A}$.

The juxtaposition $\widetilde{\gamma}$ of the paths $p^{-1}|_{\widetilde{V}_{R}}\circ \gamma|_{(-\infty, -t^{\ast}]}$, $  \widetilde{\beta} $ and  $  p^{-1}|_{g\widetilde{V}_{A}}\circ \gamma|_{ [t^{\ast},\infty)} $
is a lift of $\gamma$ such that $\alpha(\widetilde{\gamma})\subseteq \widetilde{R}$,   $\omega(\widetilde{\gamma})\subseteq g\widetilde{A}$, see Figure \ref{fig:my_label}. The uniqueness of $\widetilde{\gamma}$ follows from the uniqueness of each one of the these paths. 
\end{proof}

\begin{figure}[h!t]
    \centering
    \includegraphics[scale=0.7]{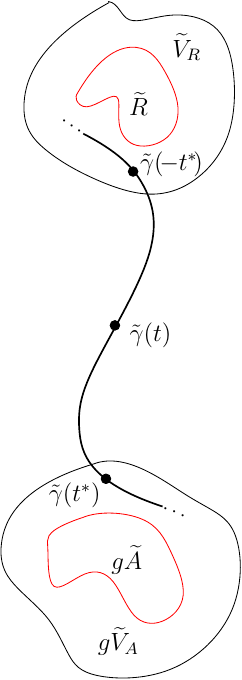}
    \caption{Construction of a lift  $\widetilde{\gamma}$ of the orbit $\gamma$.}
    \label{fig:my_label}
\end{figure}

\vspace{0.3cm}

The assumption that the orbit $\gamma$ is contained in a compact set  is necessary in the proof of Theorem \ref{teo:SVL}. Figure \ref{fig:Contra_ex_4} shows an orbit which is not contained in any compact set. Hence, one can not apply Theorem \ref{teo:SVL}.

\begin{figure}[h!t]
    \centering
   \includegraphics[scale=.7]{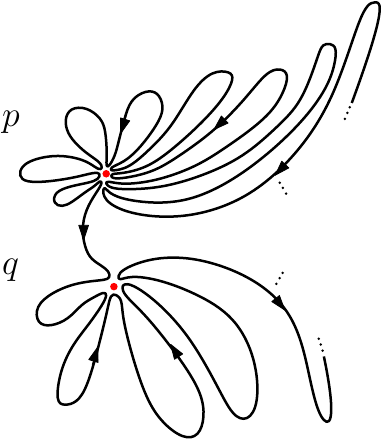} 
    \caption{Example of an orbit which is not contained in a compact set.}
    \label{fig:Contra_ex_4}
\end{figure}

\begin{definition} Let $(A,R)$ be a $p$-attractor-repeller pair for $S$ where $S$ is an invariant set. Fix  sheets $\widetilde{R}$ and $\widetilde{A}$ over $R$ and $A$, respectively.
\begin{enumerate}
    \item Given $g \in G$, an orbit $\gamma\subset C^*(R,A)$ is said to be  a $g$-orbit if there is a lift   $\widetilde{\gamma}$ of $\gamma$ such that  $\alpha(\widetilde{\gamma})\subseteq \widetilde{R }$ and  $\omega(\widetilde{\gamma})\subseteq g\widetilde{A}$. The union of all $g$-orbits between $R$ and $A$ is denoted by $C_g(R,A)$.

    \item For each $g\in G$, define
     $${S}_{R,gA}= R \cup C_g(R,A)  \cup A
 \qquad   \text{and} \qquad
\widetilde{S}_{{R},g{A}}:=\widetilde{R}\cup  C^{\ast}(\widetilde{R},g\widetilde{A})\cup g\widetilde{A}.$$

\end{enumerate}

\end{definition}

It is important to keep in mind that  
 $C_g(R,A)$, $S_{R,gA}$ and $\widetilde{S}_{R,gA}$ depend on
the choice of the sheets $\widetilde{R}$ and $\widetilde{A}$.
Note that,  whenever $g\neq h$, one has $ {S}_{{R},g{A}}\cap {S}_{{R},h{A}} =  R\cup A.$
 By definition, a $g$-orbit always has nonempty $\alpha$- and $\omega$-limit sets.

It is clear that an invariant set $S$ is not necessarily evenly covered, i.e. $p^{-1}(S)$ is not  necessarily a union of disjoint invariant sets homeomorphic to $S$, even though this property  holds, by assumption, for $p^{-1}(A)$ and $p^{-1}(R)$. However, the next theorem gives a sufficient condition for ${S}_{{R},gA}$ to be evenly covered.

\begin{theorem}\label{teo:Marimon}

Let $(A,R)$ be a $p$-attractor-repeller pair for an invariant set $S$. Fix sheets $\widetilde{R}$ and $\widetilde{A}$ over $R$ and $A$, respectively. Given $g\in G$,  if  $\widetilde{S}_{{R},g{A}}$ is compact,
 then 
 $p^{-1}({S}_{{R},gA})=\displaystyle\bigsqcup_{h\in G} h\widetilde{S}_{{R},g{A}}$ and, for all $h\in G$,
 $p|_{h \widetilde{S}_{{R},g{A}}}:h \widetilde{S}_{{R},g{A}}\to {S}_{{R},g{A}}$ is a homeomorphism.
\end{theorem}

In order to  prove Theorem \ref{teo:Marimon},  one establishes  some properties of $\widetilde{S}_{{R},g{A}}$ in the next lemma.

\begin{lemma}\label{prop:SVL1}Let $(A,R)$ be a $p$-attractor-repeller pair for an invariant set $S$. Fix sheets $\widetilde{R}$ and $\widetilde{A}$ over $R$ and $A$, respectively. Given $g\in G$, one has that:

\begin{enumerate}
    \item  if $\tilde{x} \in C^{\ast}(\widetilde{R},g\widetilde{A})$ then  $h\tilde{x} \in C^{\ast}(h\widetilde{R},hg\widetilde{A})$, for every $h\in G$;

\item $  h\widetilde{S}_{{R},g{A}}=\widetilde{S}_{h{R},hg{A}}$,  for all $h\in G$;
 
\item if $\tilde{x} \in C^{\ast}(\widetilde{R},g\widetilde{A})$ then   $h\tilde{x} \notin C^{\ast}(\widetilde{R},g\widetilde{A})$, for all $h\in G$ with $h\neq e$. 
 \end{enumerate}
\end{lemma}
\begin{proof}
Item (1) follows from the fact that each element $h\in G$ is a covering space equivalence and therefore it commutes with the dynamics. Items (2) and (3) follow directly from item (1).
\end{proof}

\vspace{0.3cm}

\begin{proof}[Proof of Theorem \ref{teo:Marimon}]

Since  $\widetilde{S}_{{R},g{A}}$  is compact, each $g$-orbit is  contained in a  compact set, namely ${S}_{{R},g{A}}$, hence every $g$-orbit is under the hypothesis of Theorem \ref{teo:SVL}.  
Moreover,  $h\widetilde{S}_{{R},g{A}}$  is compact for all $h\in G$ and $p|_{h\widetilde{S}_{{R},g{A}}}$ is onto, by Theorem \ref{teo:SVL}. Thus, in order to prove that $p|_{h\widetilde{S}_{{R},g{A}}}:h\widetilde{S}_{{R},g{A}}\to {S}_{{R},g{A}}$ is a homeomorphism, it is enough to prove that $p|_{\widetilde{S}_{{R},g{A}}}$ is injective. 
Let $\tilde{x}$ and $\tilde{y}$ be distinct points in $\widetilde{S}_{{R},g{A}}$. One has the following cases to consider.
\begin{enumerate}
    \item If $\tilde{x} $ or $\tilde{y}$ belong to $\widetilde{R}$ or $\widetilde{A}$, then it is straightforward that $p(\tilde{x})\neq p(\tilde{y})$.
    
    \item If  $\tilde{x}$ and  $\tilde{y}$ do not belong to $\widetilde{R}$ nor $\widetilde{A}$,  then there are two possibilities:
 \begin{enumerate}

     \item  $\tilde{x} $ and  $\tilde{y}$ belong to the same connecting orbit  in $C^{\ast}(\widetilde{R},g\widetilde{A}) $, which is an aperiodic orbit. Then there is a one-to-one correspondence between this orbit and its projection via $p$, by Theorem \ref{teo:SVL}. Therefore, $p(\tilde{x})\neq p(\tilde{y})$.
     
     \item  $\tilde{x} $ and  $\tilde{y}$ belong to different connecting orbits  in $C^{\ast}(\widetilde{R},g\widetilde{A}) $. Suppose that $p(\tilde{x})= p(\tilde{y})$. Thus  there exists ${h'}\in G$ such that $\tilde{x} = {h'}\tilde{y}$.
     By  Lemma \ref{prop:SVL1}, ${h'}=e$. Hence $\tilde{x}$ and $\tilde{y}$ belong to the same orbit, which is a contradiction.   

 \end{enumerate}   
 
In all cases one verifies that $p(\tilde{x})\neq p(\tilde{y})$, therefore $p|_{h \widetilde{S}_{{R},g{A}}}$ is injective.

\end{enumerate}
By Lemma \ref{prop:SVL1}, it follows that $p^{-1}({S}_{{R},g{A}})=\displaystyle\bigsqcup_{h\in G} h \widetilde{S}_{{R},g{A}}$.   
\end{proof}

The next proposition is a direct consequence of Theorem \ref{teo:Marimon} and Theorem \ref{prop_W}.

\begin{proposition}\label{cor:IISIFF} 
Let $(A,R)$ be a $p$-attractor-repeller pair for an invariant set $S$. Fix sheets $\widetilde{R}$ and $\widetilde{A}$ over $R$ and $A$, respectively. Given $g\in G$,  if  $\widetilde{S}_{{R},g{A}}$ is compact, then ${S}_{{R},g{A}} $ is an isolated invariant set if and only if $\widetilde{S}_{{R},g{A}} $ is an isolated invariant set. In this case, $CH_{\ast}(\widetilde{S}_{R,gA}) = CH_{\ast}(S_{R,gA})$.
\end{proposition}

In general, $p^{-1}(S)$ is not an isolated invariant set even when $S$ is an isolated invariant set. However,  whenever $S$ is an isolated invariant set and  $\widetilde{S}_{{R},g{A}}$ is compact for all $g\in G$, the next result guarantee  that   $p^{-1}(S)$ can be decomposed into a union of isolated invariant sets.

\begin{theorem}\label{teo:G=G'} Let $S $ be an isolated invariant set and  $(A,R)$ a $p$-attractor-repeller pair of $S$.
Fix sheets $\widetilde{R}$ and $\widetilde{A}$ over $R$ and $A$, respectively. Given $g\in G$,  if  $\widetilde{S}_{{R},g{A}}$ is compact,
then ${S}_{{R},g{A}} $ is an isolated invariant set.
\end{theorem}
\begin{proof}

Since $\widetilde{S}_{{R},g{A}}$ is compact,
then ${S}_{{R},g{A}} $ is compact. Clearly ${S}_{{R},g{A}} $ is an invariant set. Hence, one needs to prove that
$S_{R,gA}$ is isolated.

By Theorem \ref{teo:Marimon},
 $p|_{ \widetilde{S}_{{R},g{A}}}: \widetilde{S}_{{R},g{A}}\to {S}_{{R},g{A}}$ is a homeomorphism, which can be extended to   a homeomorphism   $p|_{\widetilde{V}}: \widetilde{V}\to V$   where $\widetilde{V}$ is a compact neighborhood  of $ \widetilde{S}_{{R},g{A}}$, by  Theorem \ref{prop_W}.

Given $h \in G$, let $\widetilde{U}_{h}$ be an open neighborhood of $h\widetilde{A} $ such that $\omega(\widetilde{U}_h)= h\widetilde{A}$.
Now, consider  the set $$\widetilde{W} = \widetilde{V}\setminus \bigcup_{h \in G\backslash\{g\}}\widetilde{U}_h   $$
which is still a compact neighborhood of 
$ \widetilde{S}_{{R},g{A}}$, since no $g$-orbit intersects the sets $\widetilde{U}_h$, for $h\neq g$. 

Let $N$ be an isolating neighborhood of $S$. Then $N\cap p(\widetilde{W})$ is a compact neighborhood of ${S}_{R,gA}$  and ${S}_{R,gA}$ is the maximal invariant set in $N\cap p(\widetilde{W})$.
Therefore, ${S}_{R,gA}$ is an isolated invariant set.
\end{proof}

In general, one has $\displaystyle\bigcup_{g\in G}{S}_{{R},g{A}}\subsetneq S$.   However, if $S$ is a compact set (or an isolated invariant set), by Theorem \ref{teo:SVL} the equality holds and $p^{-1}(S)=\displaystyle\bigcup_{g,h \in G} h\widetilde{S}_{{R},g{A}} $.  Moreover, if  $\widetilde{S}_{{R},g{A}}$ is compact,
 Theorem \ref{teo:G=G'} and Proposition \ref{cor:IISIFF}  guarantee that $\widetilde{S}_{{R},g{A}}$ is  an isolated invariant set and $p^{-1}(S)$ is a disjoint union of isolated invariant sets.

In the remainder of this subsection,  one establishes some results for the case that $\displaystyle S=\bigcup_{g\in G}{S}_{{R},g{A}}$.

\begin{lemma}\label{lemma:SclosedSgclosed}Let $S$ be an invariant set which  admits a $p$-attractor-repeller pair  $(A,R)$ such that  $S=\displaystyle\bigcup_{g\in G}{S}_{{R},g{A}} $. 

 Given $g\in G$,  if $S$ is a closed set then $\widetilde{S}_{{R},g{A}}$ is also closed.
\end{lemma}
\begin{proof}
Suppose that $\widetilde{S}_{R,gA}$ is not closed. Let $y \in cl(\widetilde{S}_{R,gA})\setminus  \widetilde{S}_{R,gA}$.  Since $S$ is closed,  $p^{-1}(S)$ is closed, hence $y \in p^{-1}(S)$.  Clearly  $y \notin \widetilde{R}$, $y \notin  p^{-1}(A) $. 
By hypothesis, there exists $h \in G$ such that $y \in \widetilde{S}_{R,hA}$. Of course $h \neq g$ and there is  $t' \in \mathbb{R}$ such that $\widetilde{\varphi}(t',y)\in \widetilde{U}_{h}$, where $\widetilde{U}_{h}$ is an open neighborhood of $h\widetilde{A} $ such that $\omega(\widetilde{U}_h)= h\widetilde{A}$. Let $\widetilde{B}\subset \widetilde{U}_{h}$ be a neighborhood of $\widetilde{\varphi}(t',y)$ and $T=\widetilde{\varphi}([-t',0],\widetilde{B})$.
Let $(x_n)$ be a sequence in $\widetilde{S}_{R,gA}\setminus (\widetilde{R} \cup g\widetilde{A})$ converging to $y$. 
For $n$ sufficiently large, $x_n \in T$, hence $\omega(x_n) \subset h\widetilde{A}$. On the other hand, $x_n \in \widetilde{S}_{R,gA}\setminus (\widetilde{R} \cup g\widetilde{A})$, which means that $\omega(x_n) \subset g\widetilde{A}$. That is  a contradiction since $g \neq h$.
\end{proof}

It follows from  Lemma \ref{lemma:SclosedSgclosed} that Theorem \ref{teo:G=G'} holds
assuming a weaker hypothesis:
$\widetilde{S}_{R,gA}$ is contained in a compact set of $\widetilde{X}$. In particular, it holds when $\bigcup_{g\in G}\widetilde{S}_{{R},g{A}}$ is compact.

Considering  sheets $\widetilde{R}$ over $R$ and $\widetilde{A}$ over $A$, define $$\widetilde{C}_R^S:=\displaystyle\bigcup_{g\in G}\widetilde{S}_{{R},g{A}}.$$ Note that $p(\widetilde{C}_R^S)=S$ and $p$ restricted to $\widetilde{C}_R^S\setminus p^{-1}(A)$ is a homeomorphism.

The next proposition gives a condition in order to guarantee that the set of all $g\in G$ such that $C_g(R,A) \neq \emptyset$ is finite, i.e. $\widetilde{C}_R^S$ can be written as a finite union of sets  $\widetilde{S}_{{R},g{A}}$.

\begin{proposition}\label{prop:fin}
Let $S$ be an invariant set which  admits a $p$-attractor-repeller pair  $(A,R)$ such that  $S=\displaystyle\bigcup_{g\in G}{S}_{{R},g{A}} $. 
Then $\widetilde{C}_R^S$ is compact if and only if  $\widetilde{S}_{{R},g{A}}$ is a compact invariant set, for all $g\in G$ and there exists a finite subset $\Upsilon $ of $G$ such that $\widetilde{C}_R^S=\displaystyle\bigcup_{g\in \Upsilon}\widetilde{S}_{{R},g{A}}$.
\end{proposition}
\begin{proof} 
If there exists $\Upsilon \subseteq G$  finite then $\widetilde{C}_R^S$ is a finite union of compact sets, hence $\widetilde{C}_R^S$ is compact.
On the other hand, assume that $\widetilde{C}_R^S$ is  a compact set. It follows that $S$ is a compact set and, by Lemma \ref{lemma:SclosedSgclosed},   $\widetilde{S}_{{R},g{A}}$ is a compact invariant set, for all $g\in G$.  Now, suppose that it does not exist a finite  set  $\Upsilon \subseteq G$ such that $\widetilde{C}_R^S=\displaystyle\bigcup_{g\in \Upsilon}\widetilde{S}_{{R},g{A}}$. Let $(x_n)$ be a sequence of points such that $x_n\in  g_n \widetilde{A}$. Since $A$ is evenly covered, the sequence $x_n$ does not have any accumulation point, which contradicts the fact that $\widetilde{C}_R^S$  is compact. 
\end{proof}

\begin{corollary}
 
Let $S$ be an invariant set which  admits a $p$-attractor-repeller pair  $(A,R)$ such that  $S=\displaystyle\bigcup_{g\in G}{S}_{{R},g{A}} $. 

If $\widetilde{C}_R^S$ is compact then there exists a finite subset $\Upsilon $ of $G$ such that ${S}=\displaystyle\bigcup_{g\in \Upsilon}{S}_{{R},g{A}}$.
\end{corollary}

In the next subsection, one introduces $p$-connection matrices for invariant sets. Note that even when $S$ is not an isolated invariant set but $\widetilde{S}_{{R},g{A}} $ is an isolated invariant set for all $g\in G$, then $p^{-1}(S)$ can still be decomposed into a union of isolated invariant sets. One defines a connecting map for this general setting.
Proposition \ref{cor:IISIFF} and Theorem \ref{teo:G=G'} guarantee that whenever $S$ is an isolated invariant set and $\widetilde{S}_{{R},g{A}}$ is compact then $\widetilde{S}_{{R},g{A}}$ is an isolated invariant set.
Hence, for this particular case, the connecting map for a p-attractor-repeller pair is well defined, as proved in the next section.

\begin{remark}\textcolor[rgb]{1,1,1}{.}
 In \cite{MR936812}, McCord decomposed  the  set of connections $C(R,A)$ for an isolated invariant set in a topological manner. Herein we decompose it by taking into account the covering action. Moreover, we do not require that $S$ is an isolated set.
\end{remark}

\subsection{$p-$Connection Matrices for $p$-Attractor-Repeller Decompositions}\label{subsec:4.2}

In this subsection, we define a boundary map that ``counts'' the flow lines between a repeller $R$ and an attractor $A$ by means of the lifts of these connections via the covering map $p$. In order to accomplish that, one needs to have at hand an algebraic structure which makes it possible to  ``count'' these flow lines in a suitable way.  In what follows, we define this structure, denoted by $\mathbb{Z}((G))$, where $G$ is the deck transformation group associated to $p$.
\begin{enumerate}
\item[(H-1)]  If $G$ is a finite group, we consider  $\mathbb{Z}((G))$ as the group ring $\mathbb{Z}[G]$.

\item[(H-2)] Assume that $G$  is  a totally ordered group, i.e. $G$ is equipped with a total ordering $\leq$ that is compatible with the multiplication of $G$, (for all $x,y,z\in G$, $x \leq y$ implies that $zx\leq zy$ and 
$xz \leq yz$). Moreover, assume that the set 
$\{g\in G \mid C_{g}(R,A)\neq\emptyset\}$ is well-ordered with respect to the order $\leq$.
 For every formal series $\eta \in \mathbb{Z}[G]$ $$\eta=\sum_{g\in G}a_gg \ , $$
where $a_g\in \mathbb{Z}$, the support of $\eta$ is defined as
$$supp(\eta) = \{ g\in G \mid a_g\neq 0 \}.$$
Let $\mathbb{Z}((G))$ be the ring of the formal series on $G$ that have a well-ordered support. For more details see \cite{MR3133710}.

\item[(H-3)] For more general $G$, we will assume that $\#\{g\in G;\,\, C_g(R,A)\neq\emptyset\}<\infty$ and $\mathbb{Z}((G))=\mathbb{Z}[G]$. 
        
    \end{enumerate}

An important particular case of (H-2) is when $G$ is an infinite cyclic group, namely $G=<g>$. In this case,   $\mathbb{Z}((G))$ is the Novikov ring $ \mathbb{Z}((t))$, as defined in Subsection \ref{subsec:Nov}.

Note that any totally ordered group is torsion-free. The converse holds for abelian groups, i.e.,  an abelian group admits a total ordering if and only if it is torsion-free. 

The conditions on the flow  in (H-2) and (H-3) are imposed to guarantee that there are no bi-infinite connections and this fact is necessary in order to have a  well defined boundary map.

Let $S$ be an invariant set and $(A,R)$ a $p$-attractor-repeller pair  of $S$. Fix sheets $\widetilde{R}$ and $\widetilde{A}$ over $R$ and $A$, respectively.
Consider the subset $G'$ of $G$ of all elements $g\in G$ such that  $\widetilde{S}_{{R},g{A}}$ is an  isolated invariant set.
By Proposition \ref{cor:IISIFF}, one has that  $S_{{R},g{A}} = {R}\cup  C_g({R},{A})\cup {A}$ is also an isolated invariant set.

Clearly,  $(g\widetilde{A},\widetilde{R})$ is an attractor-repeller pair for $\widetilde{S}_{{R},g{A}}$, 
for each $g\in G'$, thus
the homology Conley exact sequence of the pair $(g\widetilde{A},\widetilde{R})$ is
\begin{equation}\label{eq:exactpairsequence}
\dots \stackrel{}{\longrightarrow} CH_{\ast}(g\widetilde{A}) \stackrel{i_{\ast}}{\longrightarrow} CH_{\ast}(\widetilde{S}_{{R},g{A}})  \stackrel{p_{\ast}}{\longrightarrow} CH_{\ast}(\widetilde{R})  \xrightarrow{\widetilde{\delta}_{\ast}(\widetilde{R}, g \widetilde{A})} CH_{\ast-1}(g\widetilde{A})  {\longrightarrow} \cdots. 
\end{equation} 
One can build up  an analogous exact sequence for any pair $(g \widetilde{A},h\widetilde{R})$ whenever $h^{-1}g \in G'$.
By the equivariance, one has  that  $\widetilde{S}_{g{R},g{A}}\cong\widetilde{S}_{{R},{A}}$  and $\widetilde{S}_{h{R},g{A}}\cong \widetilde{S}_{{R},h^{-1}g{A}}$, hence $\widetilde{\delta}_{\ast}(\widetilde{R},  \widetilde{A}) = \widetilde{\delta}_{\ast}(g \widetilde{R}, g \widetilde{A})$
 and $\widetilde{\delta}_{\ast}(h\widetilde{R}, g \widetilde{A}) = \widetilde{\delta}_{\ast}( \widetilde{R}, h^{\!-\!1}g \widetilde{A})$.

Fix the sets  $B_k(R)=\{r_{\alpha}^k,\ \alpha\in \Lambda_k\}$ and $B_k(A)=\{a_{\alpha}^k,\ \alpha\in \Gamma_k\}$ of generators for $CH_k({R})$ and  $CH_k({A})$, respectively. 
Using the isomorphisms $CH_{\ast}(h\widetilde{R})= CH_{\ast}(R)$ (resp., $CH_{\ast}(h\widetilde{A})= CH_{\ast}(A)$), as in Proposition \ref{prop:Ewertonmon}, one can  define 
\begin{eqnarray}\label{defbordoAR}
 {\delta}^N_k(R,A)   :  \mathbb{Z}((G)) \otimes_{\mathbb{Z}[G]}  \mathbb{Z}[G][B_k(R)]  & \longrightarrow &   \mathbb{Z}((G))   \otimes_{\mathbb{Z}[G]}   \mathbb{Z}[G][B_{k-1}(A)]  \nonumber\\
 h \otimes  r_\alpha^k& \longmapsto   &   \sum_{g\in G} h^{\!-\!1}g  \otimes \widetilde{\delta}_k(h\widetilde{R}, g \widetilde{A})(r_\alpha^k)  \\
hr_{\alpha}^k & \longmapsto &   \sum_{g\in G}h^{\!-\!1}g \ \widetilde{\delta}_k(h\widetilde{R}, g \widetilde{A})(r_{\alpha}^k)   \nonumber 
\end{eqnarray}
where $\mathbb{Z}[G][B_k(R)] $ is a free $\mathbb{Z}[G]$-module generated by $B_k(R)$ and $\widetilde{\delta}(h\widetilde{R}, g \widetilde{A})(r^k_\alpha)$ is the null map if $h^{-1}g \notin G'$. Note that there is an injective homomorphism from $CH_k(R)$ to $\mathbb{Z}[G][B_k(R)]$ given by the map $$\begin{array}{rcl}
\xi: CH_k(R)& \longrightarrow & \mathbb{Z}[G][B_k(R)]\\ tr_\alpha^k & \longmapsto & ter_\alpha^k\end{array}$$
where $t\in \mathbb{Z}$, $e\in G$ is the identity element and $r^k_\alpha\in B_k(R)$.

Denoting $NCH_{k}(A) =  \mathbb{Z}((G)) \otimes_{\mathbb{Z}[G]} \mathbb{Z}[G][B_k(A)] $ and $ NCH_{k}(R) =  \mathbb{Z}((G))  \otimes_{\mathbb{Z}[G]} \mathbb{Z}[G][B_k(R)] $,
the map $$N\Delta: NCH_{\ast}(A) \bigoplus NCH_{\ast}(R) \longrightarrow  NCH_{\ast}(A) \bigoplus NCH_{\ast}(R)$$ defined by the matrix 
$$
\left(
\begin{array}{cc}
0 & \delta^N (R,A) \\
0 & 0 \\
\end{array}
\right)
$$
is an upper triangular boundary map and it is called a {\it $p$-connection matrix for the $p$-attractor-repeller decomposition of $S$}. Denoting $NC(S) = NCH_{\ast}(A) \bigoplus NCH_{\ast}(R)$, one has that   $(NC(S),N\Delta)$ is a chain complex.

Whenever $S$ is an isolated invariant set and $\widetilde{S}_{R,gA}$ is compact for all $g\in G$, by Theorem \ref{teo:G=G'} and Proposition \ref{cor:IISIFF}, one has that $G=G'$. Hence, the exact sequence in (\ref{eq:exactpairsequence}) is well defined  for all $g\in G$. Therefore $\delta^N$ keeps track of all information on connections between adjacent invariant sets.

The next result shows that the entry $\delta^N(R,A)$ of a $p$-connection matrix  gives dynamical information about the connecting orbits from the repeller $R$ to the attractor $A$.

\begin{proposition}  If  ${\delta}^N(R, A)$ is non-zero, then $C^{\ast}({R},{A}) \neq \emptyset$. 
\end{proposition}
\begin{proof}
 Suppose that $C^{\ast}({R},{A})=\emptyset$. Then $C(h{\widetilde{R}},g\widetilde{A})=\emptyset$, for each $g,h\in G^{\prime}$, and hence  $CH(\widetilde{S}_{h R,g A})= CH_{\ast}(g\widetilde{A}) \oplus CH_{\ast}(h\widetilde{R})$. It follows from the exactness of the long exact sequence in (\ref{eq:exactpairsequence}) that $\widetilde{\delta}_{\ast}(h\widetilde{R}, g \widetilde{A})=0$ for each $g,h\in G^{\prime}$. Therefore,  $\delta^N(R,A)=0$.  
\end{proof}

\begin{theorem}[Invariance of the $p$-connection matrices]
The chain complex $(NC(S),N\Delta)$ is invariant under equivalent regular covering spaces.
\end{theorem}

\begin{proof}
Let $p_i:\widetilde{X}_{i} \rightarrow X$ be equivalent regular covering spaces of $X$ and $G_i$ be the deck transformation groups of $p_i$, for $i=1,2$.  Given $S \subset X$ an invariant set, a pair  $(A,R)$ is a $p_1$-attractor-repeller pair of $S$ iff  it is a $p_2$-attractor-repeller pair of $S$.  

Let $h:\widetilde{X}_2\rightarrow \widetilde{X}_1$ be a  homeomorphism which provides an equivalence of the covering spaces and
fix sheets $\widetilde{R},\widetilde{A}\subset \widetilde{X}_{1}$ over $R$ and $A$ with respect to $p_1$.
 Given $g\in G_1$, 
it follows from Proposition \ref{prop:equivariant} that  $h^{-1}(\widetilde{S}_{R,gA})= h^{-1}(\widetilde{R})\cup  C(h^{-1}(\widetilde{R}),h^{-1}(\widetilde{A}))\cup h^{-1}(\widetilde{A})$  is an isolated invariant set in $\widetilde{X}_2$ if and only if $\widetilde{S}_{R,gA}$ is an isolated invariant set in $\widetilde{X}_1$.

Let $\mathfrak{h}$ be the isomorphism between   $G_1$ and $G_2$ induced by $h$. One has that $\mathfrak{h}$ induces an isomorphism 
$$H: \mathbb{Z}((G_2)) \otimes_{\mathbb{Z}[G_2]} \bigg( \mathbb{Z}[G_2][B_k(A)] \oplus \mathbb{Z}[G_2][B_k(R)]\bigg) \longrightarrow
\mathbb{Z}((G_1)) \otimes_{\mathbb{Z}[G_1]}\bigg( \mathbb{Z}[G_1][B_k(A)]\oplus \mathbb{Z}[G_1][B_k(R)]\bigg).
$$
  Note that, $H$ commutes with the boundary map, $H\circ \delta^N_2=\delta^N_1 \circ H$, since $h_\ast \circ \widetilde{\delta}_{2\ast} =\widetilde{\delta}_{1\ast}\circ h_\ast$, where $\widetilde{\delta}_{i\ast}$ is the connection map in the long exact sequence in  (\ref{eq:exactpairsequence}). Then the chain complexes that arise from the covering spaces $(\widetilde{X}_1,p_1)$ and $(\widetilde{X}_2,p_2)$ are isomorphic. 
\end{proof}

In the case that $S$  is an isolated invariant set and $p$ is the trivial covering map,  $p$-connection matrices coincide with the classical connection matrices defined by Franzosa in \cite{MR978368}. In this sense, the $p$-connection matrix theory, developed herein, generalizes the classical connection matrix theory.
More specifically,  when one considers the trivial covering action or one projects the group $G$ to the trivial group ($g\mapsto e$), the p-connection matrix is the usual connection matrix as one proves in the next result.

\begin{theorem}\label{thm_proj}
Let $S$ be an isolated invariant set and $(A,R)$ a p-attractor-repeller pair of $S$ associated to a covering map $p$.
If $\widetilde{C}_R^S$ is compact,
then the following diagram commutes 
$$\xymatrix{
NCH(A)\oplus NCH_\ast (R) \ar[r]^{N\Delta} \ar[d]^\Pi& NCH(A)\oplus NCH (R) \ar[d]^\Pi\\
CH(A)\oplus CH(R) \ar[r]^{\Delta}& CH(A)\oplus CH(R)
}
$$
where $\Pi$ is the following projection induced by the covering map $p$:
$$\begin{array}{rcl}
 \Pi:NCH_\ast(A)\oplus NCH_\ast (R)    &\longrightarrow & CH_\ast(A)\oplus CH_\ast (R) \\
    g\otimes r & \longmapsto & r
\end{array}
$$
\end{theorem}

\begin{proof}
	Fix  sheets $\widetilde{R}$ and $\widetilde{A}$ over $R$ and $A$, respectively, and let 
	$\widetilde{C}_R^S=\bigcup \widetilde{S}_{R,gA}$, where the union is over all $g\in G$ such that $ C^{\ast}(\widetilde{R},g\widetilde{A}) \neq \emptyset$. Since $\widetilde{C}_R^S$ is compact
	 then, by Proposition \ref{prop:fin}, there exist $g_1,\dots,g_n\in G$ such that $\widetilde{C}_R^S=\bigcup_{i=1}^n \widetilde{S}_{R,g_iA}$.

Also $\widetilde{C}_R^S$ is an isolated invariant set. In fact, suppose  that $\widetilde{C}_R^S$ is not an isolated invariant set and let $N$ be an isolating neighborhood for $S$.  Let $N'$ be a compact neighborhood of $\widetilde{C}_R^S$ and   $N''=N'\cap p^{-1}(N)$. Since $p^{-1}(N)$ is closed, then  $N''$  is a compact neighborhood of $\widetilde{C}_R^S$. By assumption, $\widetilde{C}_R^S$ is not an isolated invariant set, hence there is $\widetilde{x}\in N''$ such that $\widetilde{x}\cdot \mathbb{R}\subseteq N''$ and $\widetilde{x}\cdot \mathbb{R} \cap \widetilde{C}_R^S=\emptyset$. Therefore $p(\widetilde{x})\cdot \mathbb{R}\subseteq N$ and $S\subseteq p(N'')\subseteq N$, which implies that $S$ is not the maximal invariant set in $N$.  This is a contradiction.
	
Consider the flow order  $<$ for the Morse decomposition $\mathcal{M}(\widetilde{C}_R^S)=\{ g_1\widetilde{A}, \ldots, g_n\widetilde{A}, \widetilde{R}\}$ of $\widetilde{C}_R^S$. Thus there is an index filtration $\{\widetilde{N}_0, \widetilde{N}_1, \cdots \widetilde{N}_n, \widetilde{N}\}$ for $(\mathcal{M}(S'), <)$ such that: $(\widetilde{N}, \widetilde{N}_0)$ is a regular index pair for $\widetilde{C}_R^S$; $(\widetilde{N}_i, \widetilde{N}_0)$ is a regular index pair for $g_i\widetilde{A}$; $(\widetilde{N}, \widetilde{N}_A)$ is a regular index pair for $\widetilde{R}$, where $\widetilde{N}_A=\bigcup_{i=0}^n \widetilde{N}_i$; and $\widetilde{N}_i\cap \widetilde{N}_j=\widetilde{N}_0$ for $i\neq j$, see \cite{MR857439}.

Define the functions $\tau, \tau_i:\widetilde{N} \rightarrow [0,\infty]$ by  $$\tau(\widetilde{x})=\left\{ \begin{array}{rl} \textrm{sup}\{t>0 \ |\ \widetilde{x} \cdot [0,t]\subseteq \widetilde{N}\backslash \widetilde{N}_A\}, & \text{ if } \widetilde{x} \in  \widetilde{N}\backslash \widetilde{N}_A\\
		0, & \text{ if } \widetilde{x} \in  \widetilde{N}_A 
	\end{array} \right.$$ and
	
	$$\tau_i(\widetilde{x})=\left\{ \begin{array}{rl} \textrm{sup}\{t>0 \ |\ \widetilde{x} \cdot [0,t]\subseteq \widetilde{N}\backslash \widetilde{N}_{i}\}, & \text{ if } \widetilde{x} \in  \widetilde{N}\backslash \widetilde{N}_{i}\\
	0, & \text{ if } \widetilde{x} \in  \widetilde{N}_{i}
	\end{array} \right. .$$

	Since  $(\widetilde{N}, \widetilde{N}_A)$ is  a regular index pair, then $\tau$  is continuous.  By  Lemma 5.2 in  \cite{MR797044}, the index pair  $(\widetilde{N}, \widetilde{N}_i)$ is also regular, hence $\tau$ and $\tau_i$ are continuous. 
	
    The connection maps between index pairs  $\delta: \widetilde{N}/\widetilde{N}_A \longrightarrow \Sigma \widetilde{N}_A/\widetilde{N}_0$ and $\delta_i: \widetilde{N}/\widetilde{N}_A \longrightarrow \Sigma \widetilde{N}_{i}/\widetilde{N}_{0}$ are defined as
	
	$$\delta([\widetilde{x}])= \left\{\begin{array}{rl} 
	[\widetilde{x} \cdot \tau(\widetilde{x}), 1-\tau(\widetilde{x})], & 0\leq \tau(\widetilde{x})\leq 1,\\ \ 
 	 [\widetilde{N}_0\times 0], & 1\leq \tau(\widetilde{x})\leq \infty,
	\end{array}\right.
	$$

		$$\delta_i([\widetilde{x}])= \left\{\begin{array}{rl} 
	[\widetilde{x} \cdot \tau_i(\widetilde{x}), 1-\tau_i(\widetilde{x})], & 0\leq \tau_i(\widetilde{x})\leq 1,\\ \ 
	[\widetilde{N}_0\times 0], & 1\leq \tau_i(\widetilde{x})\leq \infty.
	\end{array}\right.
	$$

	When  $0\leq\tau(\widetilde{x})\leq 1$, one has that $\delta([\widetilde{x}])=\delta_i([\widetilde{x}])$, where $i$ is such that  $\widetilde{x}\cdot\tau(\widetilde{x})\in \widetilde{N}_{i}$. When $\tau(\widetilde{x})\geq 1$,  then $\delta([\widetilde{x}])=\delta_i([\widetilde{x}])=\delta_j([\widetilde{x}])= [\widetilde{N}_0\times 0]$, for all $i$ and $j$.
	Since $[\widetilde{N}_0\times 0]\in \Sigma \widetilde{N}_{i}/ \widetilde{N}_{0}$ for all $i$, then there is a natural isomorphism 
	$$\Sigma \widetilde{N}_A/\widetilde{N}_0 \simeq \bigvee_i (\Sigma \widetilde{N}_{i}/  \widetilde{N}_{0}),
	$$ where ``$\vee$'' denotes the wedge sum with base point $[\widetilde{N}_0\times 0]$.

  Note that  $\delta_i(\widetilde{N}/\widetilde{N}_A)\subseteq \Sigma \widetilde{N}_{i}/ \widetilde{N}_{0}$ and  if $\widetilde{x}\cdot \tau(\widetilde{x})\in \widetilde{N}_{i}\cap \widetilde{N}_{j}$, for $i\neq j$, then $\widetilde{x}\cdot \tau(\widetilde{x})\in \widetilde{N}_{0}$. Therefore  $\delta=\vee_i \delta_i$.{\footnote{	Given maps $f:A\to B$ and $g:A\to C$, one defines $f\vee g:A \to B\vee C$ by $f\vee g:A \stackrel{f+g}{\rightarrow} A+B\to (A+B)/{a\sim b}=A\vee B$, where $+$ is the sum operation between topological spaces and $a\in A$ and $b\in B$ are base points.}} 
	
	Applying the homological functor $\textrm{H}$ on $\Sigma^{-1}\circ \delta=\Sigma^{-1} \circ \vee_i\delta_i$, we have the usual homological connection map $\delta^N=\oplus_i \delta^N_i$. By projecting with respect to the covering map $p$, we obtain $\delta(N,N_A)=\oplus_i \delta_i(N,N_A)$, since $\Pi$ is induced by $p$. Hence, the following diagram commutes 
		$$\xymatrix{
   & & \oplus_i  NCH(\widetilde{N}_{i},\widetilde{N}_{0}) \ar[d]^\simeq\\
 \cdots	\ar[r] &	NCH(\widetilde{N},\widetilde{N}_{A}) \ar[r]^{\delta^N} \ar[ru]^{ \oplus \delta_i^N}\ar[d]^{\Pi(\widetilde{N},\widetilde{N}_{A})}& NCH(\widetilde{N}_{A},\widetilde{N}_{0}) \ar[d]^{\Pi(\widetilde{N}_{A},\widetilde{N}_{0})} \ar[r]& \cdots\\
\cdots	\ar[r] &	CH({N},{N}_{A}) \ar[r]^{\delta(N,N_A)}& CH({N}_{A},{N}_{0}) \ar[r]& \cdots
	}$$ where $(N, N_A, N_0)$ is an index filtration for $(A,R)$.
	\end{proof}

\subsection{$p$-Morse Decomposition}\label{sec:p-morsedec}

Let $(P,<)$ be a partial ordered set with  partial order $<$, where $P$ is a finite set of indices.
One says that $\pi$ and $\pi'$ are {\it adjacent elements} with respect to $<$ if they are distinct and there is no element $\pi''\in P$ satisfying $\pi''\neq \pi,\pi'$ and $\pi < \pi''<\pi' $ or $\pi' < \pi''<\pi $.


In what follows, we define  $p-$Morse decomposition for an invariant set $S$ which is not necessarily isolated or even not compact.

\begin{definition}
Let $S$ be an invariant set and $(P,<)$ be a partial ordered set.
A family of disjoint  isolated invariant sets $\mathcal{M}(S)= \{ M_{\pi} \}_{\pi \in P}$ is a ($<$-ordered) $p$-Morse decomposition for $S$ if the 
 sets $M_{\pi}$ are evenly covered for all $\pi \in P$ and
given  $x\in S$, one has that either $x \in M_{\pi}$ for some $\pi\in P$ or $x \in C^{\ast}(M_{\pi'},M_{\pi})$, where  $\pi,\pi'\in P$ and $\pi<\pi'$.

\end{definition}

Each set $M_{\pi}$ is called a $p$-{\it Morse set}.
The partial order $<$ on $P$ induces an obvious partial order on $\mathcal{M}(S)$, called an admissible ordering of the $p$-Morse decomposition.
  The flow defines an
admissible ordering of $\mathcal{M}(S)$, called the \textit{flow ordering} of $\mathcal{M}(S)$, denoted $<_f$, and such
that $M_{\pi}<_f M_{\pi'} $ if and only if there exists a sequence of distinct elements
of $P: \pi = \pi_0,\ldots,\pi_n=\pi'$, where  the set of connecting orbits $C^{\ast}(M_{\pi_j},M_{\pi_{j-1}})$ between $M_{\pi_j}$ and $M_{\pi_{j-1}}$ is nonempty, for each $j = 1, \ldots ,n$. Note that every admissible ordering of $\mathcal{M}(S)$ is an extension of $<_f$.

Given two adjacent elements $\pi,\pi'$ define 
$${M}_{{\pi},{\pi}^{\prime}}={M}_{\pi}\cup  C^{\ast}({M}_{\pi},{M}_{\pi^{\prime}})\cup {M}_{\pi^{\prime}}$$
which is an invariant set. Moreover, $(M_{\pi'},M_{\pi})$ is a $p$-attractor-repeller pair for ${M}_{{\pi},{\pi}^{\prime}}$.
From now on fix sheets  $\widetilde{M}_{\pi}$ over ${M}_{\pi}$, for all $\pi \in P$.
  Consider the subset $G_{{\pi}{\pi}^{\prime}}$ of $G$ of all elements $g\in G$ such that  $$\widetilde{M}_{\pi,g{\pi}^\prime}  :=\widetilde{M}_{\pi}\cup  C^{\ast}(\widetilde{M}_{\pi},g\widetilde{M}_{\pi'})\cup g\widetilde{M}_{\pi'}$$ is an  isolated invariant set (hence, compact).   Clearly,  $(g\widetilde{M}_{\pi^{\prime}},\widetilde{M}_{\pi})$ is an attractor-repeller pair in $\widetilde{M}_{{\pi},g{\pi}^{\prime}}$ as in \cite{MR978368}. Hence, for each $g\in G_{{\pi}{\pi}^{\prime}}$ there exists a long exact sequence
$$ \dots \stackrel{}{\longrightarrow} CH_{\ast}(g\widetilde{M}_{\pi^{\prime}}) \stackrel{i_{\ast}}{\longrightarrow} CH_{\ast}(\widetilde{M}_{{\pi},g{\pi}^{\prime}})  \stackrel{p_{\ast}}{\longrightarrow} CH_{\ast}(\widetilde{M}_{\pi})  \xrightarrow{\widetilde{\delta}_{\ast}(\widetilde{M}_{\pi}, g \widetilde{M}_{\pi^{\prime}})} CH_{\ast-1}(g\widetilde{M}_{\pi^{\prime}})  {\longrightarrow} \cdots  $$

By Proposition \ref{cor:IISIFF}, ${M}_{{\pi},g{\pi}^{\prime}}$ is an isolated invariant set. 
Fix a set of generators $B_k({M}_{\pi})$  for $CH_k({M}_{\pi})$, for each $\pi \in P$.

Denoting $NCH_{k}({M}_{\pi}) = \mathbb{Z}((G))\otimes_{\mathbb{Z}[G]} \mathbb{Z}[G][B_k(M_{\pi})] $,  let  
$$N\Delta:  \bigoplus_{\pi \in P} NCH_{\ast}({M}_{\pi}) \longrightarrow  \bigoplus_{\pi \in P} NCH_{\ast}({M}_{\pi}) $$  be the map  defined by the  upper triangular  matrix 
$$ N\Delta =
\left(
\begin{array}{ccc}
&&\\
 & \delta^N({\pi},{\pi^{\prime}}) & \\
 &&\\
\end{array}
\right)_{\pi,\pi' \in P},
$$   
where $ \delta^N({\pi},{\pi^{\prime}}) $ is given by $\delta^{N}({M}_{\pi},{M}_{\pi^{\prime}})$, as  defined in (\ref{defbordoAR}), if $\pi $ and $\pi'$ are adjacent elements and it is the null map  otherwise. 
The possible nonzero entries of  $N\Delta$  are always maps from $NCH(M_{\pi})$ to $NCH(M_{\pi'})$, where $\pi$ and $\pi'$ are   adjacent elements, and they give information on the orbits connecting $M_{\pi}$ to $M_{\pi'}$.

The natural question herein is how to define a map from $NCH(M_{\pi})$ to $NCH(M_{\pi'})$ when $\pi$ and $\pi'$ are not adjacent elements which would give more information than the null map. This question is related to a generalization of the work in this paper to the case of a $p$-Morse decomposition of an invariant set. 
The first step in this direction is to study the behavior of the Morse sets $M(I)$ for any interval $I$. Since we are considering $S$ as an invariant set, not necessarily isolated,  
the description of the structure of  $M(I)$ is a delicate and difficult problem.
For instance,  $M(I)$ is not necessarily an isolated invariant set and it may not be evenly covered. We will address this problem in a future work. 
  
  However, the $p$-connection matrix defined herein is  rich enough to describe the behaviour of the connecting orbits between the Morse sets $M_{\pi}$   in the case of   a
 $p$-Morse decomposition  where each  $M_{\pi}$ is a critical point of a circle-valued Morse function, as we prove
 in Section \ref{sec:novikov}.

\subsection{Examples}\label{subsec:examples}

In this subsection we present some examples where we  describe the $p$-connection matrix $N\Delta$  for groups $G$ that satisfy (H-1), (H-2) and (H-3).

\begin{example}[Klein bottle]  Let $X$ be the Klein bottle. 
 Consider a flow on  $X$ having one repelling singularity $x$, two saddle singularities $y_1,y_2$ and one attracting singularity $z$, as  in Figure \ref{klein}, where we consider the Klein bottle as the quotient space of $[0,1]\times [0,1]$ by the relations $(0, y)\sim(1, y)$ and $(x, 0) \sim (1-x, 1)$.
 Consider the $p$-Morse decomposition where each Morse set is a singularity and the partial order is given by the flow.
 
 \begin{figure}[h!t]
    \centering
\includegraphics[scale=.7]{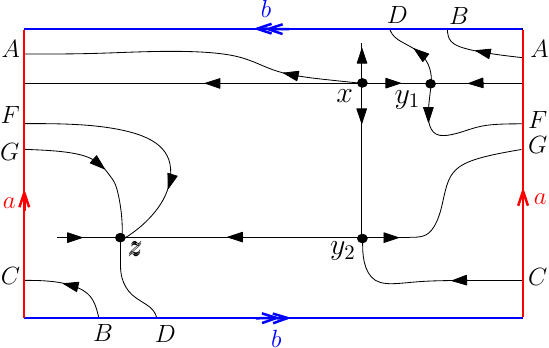}
    \caption{A flow on the Klein bottle.}
    \label{klein}
\end{figure}
 
 The universal cover  of $X$ is the plane $\mathbb{R}^2$ and its
 deck transformations  group $G$ has the presentation $\langle a, b \mid ab = b^{-1}a  \rangle$. In this case, one considers $\mathbb{Z}((G))$ as the group ring $\mathbb{Z}[G]$.

As  usual in the Morse setting,  one can associate the generator of the homology Conley index of each singularity with the singularity itself.
 With this notation, the boundary operator is  given by $\delta^N_{2}(x,y_1) = y_1 + b.y_1$, $\delta^N_{2}(x,y_2) = y_2 + bab .y_2 = y_2 + a.y_2$,  $\delta^N_{1}(y_1,z) =b. z + a.z$, $\delta^N_{1}(y_2,z) = z +  b.z$.
 
\end{example}

\begin{example}[Double torus] Consider a flow  on the double torus $X$ having the invariant set as in Figure \ref{fig:bitoro}, where we present a saddle singularity $y$, an attracting periodic orbit $\gamma_0$ and a  repeller singularity $x$.

Consider  the 5-torus $\widetilde{X}$ as   a covering space of $X$ with 4 leaves as in Figure \ref{fig:Genus_4}. The deck transformation group is $G=\mathbb{Z}_2\oplus \mathbb{Z}_2=\{a,b \mid a^2=1, b^2=1\}$, which is a finite group, hence  $\mathbb{Z}((G))$ is the group ring $\mathbb{Z}[G]$.

\begin{figure}[h!]
    \centering
\includegraphics[scale=1.5]{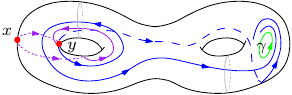}
    \caption{A flow on the double torus.}
    \label{fig:bitoro}
\end{figure}

As  usual in the Morse setting,  one can associate the generator of the homology Conley index of each singularity with the singularity itself. Let $r_1$ and $r_0$ be the  generators of $CH_{i}(\gamma)$, for $i=1$ and $i=0$, respectively. In what follows, we compute the boundary operator $\delta^N$ using this notation.  
Consider the invariant set $S=\{y\}\cup C(y,\gamma)\cup \{\gamma\}$, the boundary operator is  given by $\delta^N_{1}(y,\gamma) = a\gamma+ab\gamma $ and $\delta^N_{k}=0$ for $k\neq 1$.
 Now, consider the invariant set $S=\{x\}\cup C(x,y)\cup \{y\}$, the boundary operator is  given by
 $\delta^N_{2}(x,y) =y+ay$  and $\delta^N_{k}=0$ for $k\neq 2$.

\end{example}

\begin{figure}[h!t]
    \centering
\includegraphics[scale=.7]{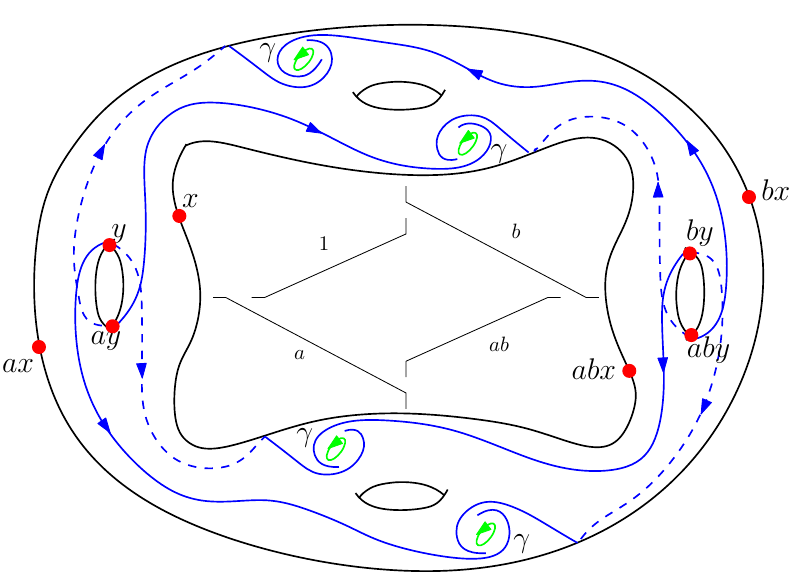}
    \caption{A covering space of the double torus with deck transformation group isomorphic to  $<a,b\ |\ a^2=b^2=aba^{-1}b^{-1}=1>\cong\mathbb{Z}_2\oplus\mathbb{Z}_2$.}
    \label{fig:Genus_4}
\end{figure}

\begin{example}[Solid double torus]
Consider a flow on the solid double torus $M$ having two consecutive critical points $p$ and $q$. Note that the Cayley graph of $\mathbb{Z}\ast\mathbb{Z}$ where every edge is a solid cylinder, as in  Figure \ref{Antenna}, is a universal covering of $M$. In this case $G=\pi_1(M)=\mathbb{Z}\ast\mathbb{Z}$.
Assume that, for each $g \in G$ there is one isolated $g$-orbit between  $p$ and $q$, and hence there are infinite isolated connections between $p$ and $q$.
This is the case where $\mathbb{Z}((G))$ satisfies the condition (H-2) in Subsection \ref{subsec:4.2}, where $\mathbb{Z}\ast\mathbb{Z}$ is equipped with the dictionary order. Note that, dynamical systems such that $\{g\in G \mid C_{g}(R,A)\neq\emptyset\}$ is infinite are in general not trivial to understand completely, however the machinery constructed in this paper contributes to have a better understanding  of the global behaviour.
\end{example}

\begin{figure}[h!t]
    \centering
\includegraphics[scale=.7]{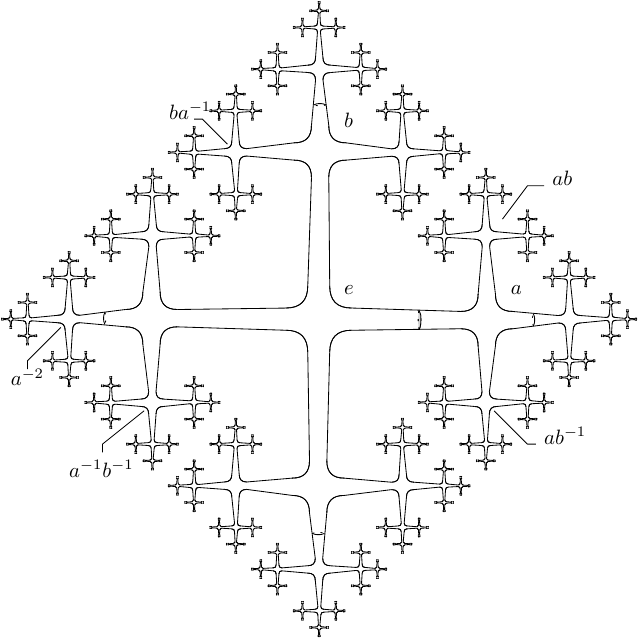}
    \caption{Universal covering of the solid double torus: Cayley graph of $\mathbb{Z}\ast\mathbb{Z}$ where every edge is a solid cylinder.}
    \label{Antenna}
\end{figure}

\newpage

\section{Novikov differential as  a $p$-Connection Matrix}\label{sec:novikov}

Let $M$ be a compact Riemannian manifold. Consider $(\overline{M},p_E)$ the infinite cyclic covering of $M$ induced by a circle-valued Morse function $f:M\rightarrow {S}^1$.
In this case, the deck transformation group is $G=\mathbb{Z}$, the group ring $\mathbb{Z}[G]$ is isomorphic to the polynomial ring $\mathbb{Z}[t]$ and $\mathbb{Z}((G))$ is isomorphic to the ring of the formal Laurent series $\mathbb{Z}((t))$. 

Consider the flow ordering $\prec_{f}$ and a ($\prec_{f}$-ordered) $p$-Morse decomposition $\mathcal{M}(M) = \{M_{\pi}\}_{\pi\in P}$ of $M$, where each $p$-Morse set $M_{\pi}$ is a critical point of $f$. 
Throughout this section, fix sheets  $\overline{M}_{\pi}$ over ${M}_{\pi}$, for all $\pi \in P$.
    
Given $M_{\pi}=\{h_k\}$ and  $M_{\pi'}=\{h_{k-1}\}$, where $h_k$ and $h_{k-1}$ are consecutive critical points of $f$ with Morse indices $k$ and $k-1$, respectively, one has that $\pi$ and $\pi'$ are adjacent elements w.r.t. $\prec_f$. Moreover
$$\overline{M}_{{\pi},t^{\ell}{\pi}^{\prime}}=\overline{M}_{\pi}\cup  C^{\ast}(\overline{M}_{\pi},t^{\ell}\overline{M}_{\pi^{\prime}})\cup t^{\ell}\overline{M}_{\pi^{\prime}}$$ 
is an isolated invariant set, for all $\ell\in \mathbb{Z}$. Hence, $G_{{\pi}{\pi}^{\prime}}=G=\mathbb{Z}$.

Consider the attractor-repeller pair $(t^{\ell}\overline{M}_{\pi^{\prime}},\overline{M}_{\pi})$ of $\overline{M}_{{\pi},t^{\ell}{\pi}^{\prime}}$, for each $\ell\in \mathbb{Z}$. There exists a long exact sequence
$$ \dots \stackrel{}{\longrightarrow} CH_{\ast}(t^{\ell}\overline{M}_{\pi^{\prime}}) \stackrel{i_{\ast}}{\longrightarrow} CH_{\ast}(\overline{M}_{{\pi},t^{\ell}{\pi}^{\prime}})  \stackrel{p_{\ast}}{\longrightarrow} CH_{\ast}(\overline{M}_{\pi})  \stackrel{\overline{\delta}_{\ast}(\overline{M}_{\pi}, t^{\ell} \overline{M}_{\pi^{\prime}})}{\xrightarrow{\hspace*{1.5cm}}} CH_{\ast-1}(t^{\ell}\overline{M}_{\pi^{\prime}})  {\longrightarrow} \dots.  $$

In this case, the set of generators $B_k(M_\pi)$ of the homology Conley index $CH_{\ast}(M_{\pi})$ has exactly one element and $B_{i}(M_\pi)=0$, for all $i\neq k$. Thus $NCH_{k}({M}_{\pi}) = \mathbb{Z}((t))  \otimes_{\mathbb{Z}[t]} \mathbb{Z}[t][B_k(M_\pi)] $ and     
$$N\Delta:  \bigoplus_{\pi \in P} NCH_{\ast}(M_{\pi}) \longrightarrow  \bigoplus_{\pi \in P} NCH_{\ast}({M}_{\pi}), $$  given by the matrix 
$$ N\Delta =
\Bigg(
\ \  \delta^{N}(\pi,\pi^{\prime})  \ \
\Bigg)_{\pi,\pi' \in P},
$$  
is an upper triangular boundary map, 
where $\delta^{N}({\pi},{\pi^{\prime}}):=\delta^{N}({M}_{\pi},{M}_{\pi^{\prime}})$ is the connecting map for the attractor-repeller pair introduced in Section \ref{subsec:4.2}.
Note that,
$\delta^{N}({\pi},{\pi^{\prime}})=0$ whenever $\pi$ and $\pi'$ are not adjacent.

\begin{example} Consider a flow $\varphi$ on the  solid torus which has two hyperbolic singularities $p$ and $q$ of indices $2$ and $1$, respectively. Moreover, for each ${\ell}\in\mathbb{Z}_+$ there is only one flow line joining $p$ and $q$ which intersects a given regular level set $\ell$ times (turns around $\ell$ times). 
Considering the invariant set  $S= \{p\} \cup C^{\ast}(\{p\},\{q\}) \cup \{q\}$, the  collection $\mathcal{M}(S) = \{ M_{\pi}=\{p\}; M_{\pi'}=\{q\} \}$ is a  ($\prec_{f}$-ordered) $p$-Morse decomposition of $S$.  See Figure \ref{fig:NovExampleinfinite}.

Fix sheets  $\overline{M}_{\pi}=\{\overline{p}\}$ and $\overline{M}_{\pi'}=\{\overline{q}\}$ over ${M}_{\pi}$ and ${M}_{\pi'}$. The set
$$\overline{M}_{{\pi},t^{\ell}{\pi}^{\prime}}=\{\overline{p}\}\cup  C^{\ast}(\{\overline{p}\},\{t^{\ell}\overline{q}\})\cup \{t^{\ell}\overline{q}\}$$
is composed by two singularities and a unique orbit between them, hence it is an isolated invariant set, for all $\ell\in \mathbb{Z}$.

Even though $\overline{C}^{S}_{\{p\}}=\displaystyle\bigcup_{\ell\in \mathbb{Z}_+}\overline{M}_{{\pi},t^{\ell}{\pi'}}$ is not compact, it can be decomposed  into a union of isolated invariant sets. Therefore $S$ is a union  of  evenly covered isolated invariant sets, i.e. ${S}=\displaystyle\bigcup_{\ell\in \mathbb{Z}_+}{M}_{{\pi},t^{\ell}{\pi'}}$.

Let $NCH_{1}(\{q\}) =  \mathbb{Z}((t)) \otimes_{\mathbb{Z}[t]} \mathbb{Z}[t][a] $ and $ NCH_{2}(\{p\}) =  \mathbb{Z}((t))  \otimes_{\mathbb{Z}[t]} \mathbb{Z}[t][r] $, where $B_1(\{q\})=\{a\}$ and  $B_2(\{p\})=\{r\}$ are generators of the homology Conley indices of $\{q\}$ and $\{p\}$, respectively.
The map $$N\Delta: NCH_{\ast}(\{q\}) \bigoplus NCH_{\ast}(\{p\}) \longrightarrow  NCH_{\ast}(\{q\}) \bigoplus NCH_{\ast}(\{p\})$$ is defined by the matrix 
$$
\left(
\begin{array}{cc}
0 & \delta^N (\{p\},\{q\}) \\
0 & 0 \\
\end{array}
\right)
$$
where 
 $ \delta^N_{2} (\{p\},\{q\}) (r)= \displaystyle  \sum_{\ell\in \mathbb{Z}_+} t^{\ell} \ \overline{\delta}_k(\overline{p}, t^{\ell} \overline{q})(r) = \displaystyle  \sum_{\ell\in \mathbb{Z}_+}  \pm t^{\ell} a $.
\end{example}

\begin{figure}[h!t]
    \centering
   \includegraphics[scale=.8]{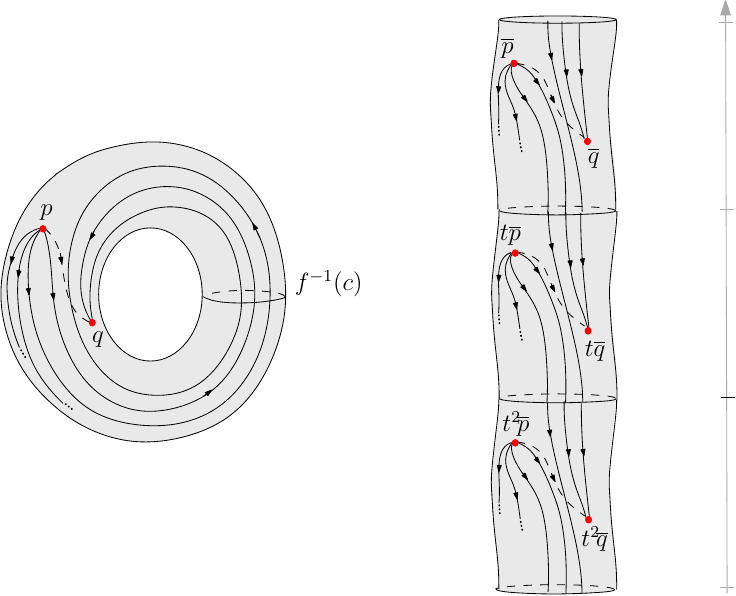} 
    \caption{A flow on a solid torus.}
    \label{fig:NovExampleinfinite}
\end{figure}

In order to prove that the Novikov boundary differential $\partial^{Nov}$  is a p-connection matrix for the Morse decomposition $\mathcal{M}(M)$, we make use of Salamon's results in \cite{MR1045282} 
 and the characterization of the Novikov complex by direct limits given in \cite{MR2017851}.

\begin{theorem}
The Novikov differential $\partial^{Nov}$ is the $p$-connection matrix for the $\prec_{f}$-ordered $p$-Morse decomposition, where $(\overline{M},p)$ is the infinite cyclic covering space.  
\end{theorem}

\begin{proof}
Assume that $1\in S^1$ is a regular value  of $f$. Denoting by $V$ the set $f^{-1}(1)\subset M$ and cutting $M$ along $V$, we obtain the fundamental cobordism $(W, tV,V)$ for $\overline M$ and the Morse function $F_V:(W,tV,V)\to ([0,1],0,1)$. 
    
\begin{figure}[h!t]
    \centering
   \includegraphics[scale=.8]{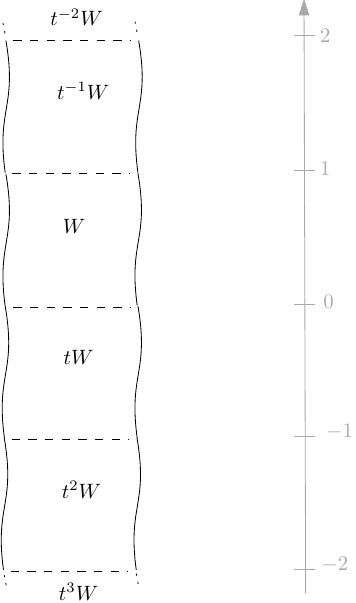} 
    \caption{A fundamental domain.}
    \label{Cornea}
\end{figure}

Note that each Morse set $M_{\pi}$, which is critical point  of $f$, has a unique lift to a  critical point of  $F_V$ in $W$ which will be denoted by $\overline{M}_{\pi}$. Moreover, one has that $$F=\bigcup_{j=-\infty}^{\infty}t^{j}F_V \ , \hspace{0.5cm}\overline{M}=\bigcup_{j=-\infty}^{\infty}t^{j}W \hspace{0.5cm} \text{and} \hspace{0.5cm} Crit_k(F)=\bigcup_{j=-\infty}^{\infty}t^{j}Crit_k(F_V).$$
    
Denote by $W(\ell) = \bigcup_{j=0}^{\ell}t^{\ell}W$ and $F(\ell) = F|_{W(\ell)}$.

Choosing the lifts of the critical points of $f$ that belong to $W$ in the construction of the chain complex $(NC(M),\Delta)$, the coefficients of the differential are in $Z[[t]] \subset Z((t))$.
    
Define $[B_k^\ell({M}_\pi)]:= [B_k(\overline{M}_\pi), B_k(t\overline{M}_\pi), B_k(t^2\overline{M}_\pi), \ldots, B_k(t^\ell\overline{M}_\pi)],$ where $B_k(t^i\overline{M}_\pi)$ is the set of generators of $CH_k(t^i\overline{M}_\pi)$.
Note that 
$\{\mathbb{Z}[B_k^{\ell}({M}_{\pi})], \pi^{\ell}_{j}\}$ is an inverse system, where $\pi^{\ell}_{j}$ are the natural projections.
Hence, the inverse limit is a based f.g. free $\mathbb{Z}[[t]]$-module and 
 $$\underset{\underset{\ell} {\longleftarrow}}{\mathrm{lim}}\ \mathbb{Z} [B_k^\ell({M}_\pi)]=\mathbb{Z}[[t]] [B_k({M}_\pi)].$$

Therefore,
\begin{eqnarray}\nonumber 
NCH_k(M_{\pi}) & = & \mathbb{Z}((t)) \otimes_{\mathbb{Z}[[t]]} \mathbb{Z}[[t]] [B_k^\ell({M}_\pi)]\\
\nonumber
& = & \mathbb{Z}((t)) \otimes_{\mathbb{Z}[[t]]} \underset{\underset{\ell} {\longleftarrow}}{\mathrm{lim}}\ \mathbb{Z} [B_k^\ell({M}_\pi)]. \nonumber
\end{eqnarray}

Now, consider the upper triangular boundary map $$ N\Delta^{\ell} =
\Bigg(
 \bar{\delta}^{\ell}(\pi,\pi') 
\Bigg)_{\pi,\pi' \in P}
$$  
where 
\begin{eqnarray}
 {\bar{\delta}}^\ell_k(\pi,\pi')   : \mathbb{Z}[B_k^\ell({M}_\pi)]  & \longrightarrow & \mathbb{Z}[B_{k-1}^\ell({M}_{\pi'})]  \nonumber\\
r_{\alpha}^k & \longmapsto &   \sum_{0\leq i\leq \ell}  \overline{\delta}_k(\overline{M}_\pi, t^i\overline{M}_{\pi'})(r_{\alpha}^k)   \nonumber 
\end{eqnarray}
is the connecting map for the attractor-repeller pair $(M_{\pi'}, M_{\pi})$, and $\bar{\delta}^\ell_k(\pi,\pi')=0$ when $\pi$ is not adjacent to $\pi'$.
Therefore, $N\Delta$ can be rewritten as the inverse limit of the maps $N\Delta^{\ell}$, i.e., for each $k\geq 0$,

$$N\Delta_k  = 1 \otimes_{\mathbb{Z}[[t]]} \underset{\underset{\ell} {\longleftarrow}}{\mathrm{lim}}  \ N\Delta^{\ell}_{k}. $$

Summarizing, 
$$ \big(NC(M),N\Delta\big) = \Bigg(\bigoplus_P \bigg(\mathbb{Z}((t)) \otimes_{\mathbb{Z}[[t]]} \underset{\underset{\ell} {\longleftarrow}}{\mathrm{lim}}\ \mathbb{Z} [B_k^\ell({M}_\pi)]\bigg) \ , \ 1\otimes_{\mathbb{Z}[[t]]} \underset{\underset{\ell} {\longleftarrow}}{\mathrm{lim}}\ N\Delta^{\ell} \Bigg). $$

On the other hand, for each $k\in  \mathbb{Z}$ and $\ell \in  \mathbb{N}$, one has that $ \mathbb{Z} [B_k^\ell({M}_\pi)]=  \bigoplus_{j=0}^\ell CH_k(t^{j}\overline{M}_{\pi})$, hence
$$NCH_k(M_{\pi})= \mathbb{Z}((t)) \otimes_{\mathbb{Z}[[t]]} \underset{\underset{\ell} {\longleftarrow}}{\mathrm{lim}}\ \bigoplus_{j=0}^\ell CH_k(t^{j}\overline{M}_{\pi}).$$

Note that $N\Delta^{\ell}_{\ast}$ coincides with the Franzosa's connection matrix,  $$\Delta^{\ell}_{k}:\bigoplus_{\pi\in P} \bigoplus_{j=0}^\ell CH_k(t^{j}\overline{M}_{\pi})\to \bigoplus_{\pi\in P} \bigoplus_{j=0}^\ell CH_{k-1}(t^{j}\overline{M}_{\pi}),$$
of the induced Morse decomposition for  $\cup_{j=0}^{\ell}t^{j}W$. 

Therefore, $N\Delta_k  =  1 \otimes_{\mathbb{Z}[[t]]} \underset{\underset{\ell} {\longleftarrow}}{\mathrm{lim}}\Delta^{\ell}_{k} $
and

$$ (NC(M),\Delta) = \Bigg(\bigoplus_P \bigg(\mathbb{Z}((t)) \otimes_{\mathbb{Z}[[t]]} \underset{\underset{\ell} {\longleftarrow}}{\mathrm{lim}}\ \bigoplus_{j=0}^\ell CH_{\ast}(t^{j}\overline{M}_{\pi})\bigg), 1\otimes_{\mathbb{Z}[[t]]} \underset{\underset{\ell} {\longleftarrow}}{\mathrm{lim}}\ \Delta^{\ell}_{\ast} \Bigg). $$

Since $W(\ell)$ is a compact manifold with no critical points in the boundary, it follows from Lemma 2 in \cite{MR1045282},  that the connection matrix for a Morse flow, given by the negative gradient of the Morse-Smale function $F(\ell)$, coincides with the Morse differential of $F(\ell)$, i.e.
$$\Delta^{\ell}_{\ast} = \displaystyle\partial_{\ast}(W(\ell), F(\ell)).$$

Since $\bigoplus_{j=0}^\ell CH_k(t^{j}\overline{M}_{\pi}) = \mathbb{Z}[Crit_k(F(\ell))]$, then
\begin{eqnarray}
   (NC(M),N\Delta) & = & \bigg( \mathbb{Z}((t)) \otimes_{\mathbb{Z}[[t]]} \underset{\underset{\ell} {\longleftarrow}}{\mathrm{lim}}\ C_{\ast}(W(\ell),F(\ell)) \ , \  1\otimes_{\mathbb{Z}[[t]]} \underset{\underset{\ell} {\longleftarrow}}{\mathrm{lim}}\ \partial_{\ast}(W(\ell), F(\ell)) \bigg)  \nonumber \\ 
    & =& \big(C^{Nov}(M,f)\ ,\ \partial^{Nov}\big) \nonumber
\end{eqnarray}
where the second equality  follows by  Lemma 2.5 in  \cite{MR2017851}.
    \end{proof}
    
    \vspace{0.5cm}
    
    As we proved in this section, the Novikov theory fits nicely as a special case of  the covering action on Conley index. Consequently, it opens the possibility to make use of a variety of tools from Conley index theory in Novikov theory.  For instance, one can study periodic orbits \cite{MR1354310,MR1879735}, chaos \cite{MR1337206}, cancellations \cite{MR3784745}, and so forth. Furthermore, it enables us to apply transition matrix as in \cite{MR3561428,MR3695843} to understand bifurcations that  may occur  when we consider a parameterized family of gradient flows of circle-valued Morse functions.

\vspace{1cm}

\section*{Acknowledgments}
The first author would like to thank  the
S\~ao Paulo Research Foundation  (FAPESP) for the
  support under grants 2020/11326-8 and 2016/24707-4.
The second author would like to thank  the
S\~ao Paulo Research Foundation  (FAPESP) for the
  support under grants
2016/24707-4
and 
2018/13481-0. The third author is affiliated with DIMACS (the Center for Discrete Mathematics and Theoretical Computer Science), Rutgers University, and IME-UFG (Instituto de Matem\'atica e Estat\'istica, Universidade Federal de Goi\'as) and would like to acknowledge the support of the National Science Foundation under grant HDR TRIPODS 1934924.

\bibliographystyle{amsplain}

\end{document}